\documentclass{article}

\usepackage{amssymb}
\usepackage{amsmath}
\numberwithin{equation}{section}
\usepackage{graphicx}

\newtheorem{defi}{Definition}[subsection]
\newtheorem{rem}{Remark}[subsection]
\newtheorem{theorem}{Theorem}[subsection]
\newtheorem{sub}{Sublemma}
\newtheorem{lemma}[theorem]{Lemma}
\newtheorem{cor}[subsection]{Corollary}
\newtheorem{prop}[subsection]{Proposition}
\newenvironment{proof}{}{\qed\bigskip

}
\newcommand\ack{\bigskip

{\bf Acknowledgement}\\}
\newcommand\qed{\hfill\ensuremath{\blacksquare}}

\newcommand\lam{\lambda}
\newcommand\Val{\mathrm{Val\,}}
\newcommand\Lam{\Lambda}
\newcommand\norm[1]{\zet{\zet{#1}}}
\newcommand\snorm[1]{\zet{\zet{#1}}}
\newcommand\htht[1]{\mathrm{top}#1}

\newcommand\val{\mathrm{val\,}}
\newcommand\ai{\mathrm{i}}
\newcommand\rbatu{R^\times}

\newcommand\ZZ{\ensuremath{\mathbf{Z}}}
\newcommand\CC{\ensuremath{\mathbf{C}}}

\newcommand\RR{\ensuremath{\mathbf{R}}}
\newcommand\NN{\ensuremath{\mathbf{N}}}
\newcommand\QQ{\ensuremath{\mathbf{Q}}}

\newcommand\vect[1]{\mbox{\boldmath$#1$}}
\newcommand\ep{\varepsilon}
\newcommand\zet[1]{\left\vert {#1} \right\vert}
\newcommand\ee{\vect{e}}

\begin{document}

\title{Integration over Tropical Plane Curves and
Ultradiscretization}
\author{Shinsuke Iwao,\\
Graduate School of Mathematical Sciences,
The University of Tokyo,\\
3-8-1
Komaba Meguro-ku, Tokyo 153-8914, Japan.\\
iwao@ms.u-tokyo.ac.jp}
\date{\today}

\maketitle

\begin{abstract}
In this article we study
holomorphic integrals on tropical plane curves
in view of ultradiscretization.
We prove that the lattice integrals over tropical curves
can be obtained as a certain limit of complex integrals over
Riemann surfaces.
\end{abstract}

\section{Introduction}

A
\textit{tropical curve} is a kind of algebraic curve
defined over the
tropical semifield $\mathbf{T}=\RR\cup\{\infty\}$
equipped with the min-plus operations:
\[
``x+y"=\min\{x,y\},\ ``xy"=x+y.
\]
The geometry over tropical curves was introduced by several authors \cite{Mik0,Spe}.
Among these works,
the theory of integration over
tropical curves was introduced
by \!Mikhalkin and Zharkov in \cite{Mik}.
According to their work,
a holomorphic differential on a tropical curve
is defined as a global section of the real cotangent sheaf
(Definition 4.1 \cite{Mik}).
Using the concept of tropical differentials,
they derive
the definition of a tropical holomorphic integral.

As one of the applications of
tropical geometry,
a number of authors have tried to solve problems
concerning 
integrable systems or dynamical systems
by using
the method of tropical geometry \cite{Takenawa,Nobe}.

The bridge between integrable systems and tropical geometry is
the method of
\textit{ultradiscretization}.
Ultradiscretization is a kind of limiting procedure,
which is usually described as 
$- \lim_{\ep\to 0^+}\ep\log{\cdot}$.
The two formulae
\begin{equation}
-\lim_{\ep\to 0^+}\ep\log{(e^{-a/\ep}\cdot e^{-b/\ep})}=a+b,\quad
-\lim_{\ep\to 0^+}\ep\log{(e^{-a/\ep}+ e^{-b/\ep})}=\min{[a,b]},
\end{equation}
($a,b\in\RR$) are fundamental.
Through ultradiscretization, we translate objects over $\CC$
into the min-plus algebra.

In this paper we study the tropicalization 
(or ultradiscretization) of 
holomorphic integrals over
complex plane curves.
Loosely speaking,
the question we wish to answer is: 
``Why is it that tropical integrals are able to
tell us something about the behaviour of complex integrals?"

Many researchers have studied the relationship between 
analytic curves and tropical curves.
Katz, Markwig and Markwig \cite{Katz} studied the $j$-invariant
of cubic curves and its tropicalization.
For the genus zero and genus one cases,
Speyer \cite{Speyer} proved the existence of an
analytic curve tropicalization of which 
coincides with a given tropical curve in any ambient space.
Helm and Katz \cite{Helm} discussed
the relationship between the tropical curve and the monodromy
action on the Hodge structure.

In order to make use of the established results for hypersurfaces
(for example Viro's approximation theorem \cite{Viro1,Viro2}),
we restrict ourselves to tropical integral calculus over plane curves
instead of considering more general tropical curves 
which have been studied by many researchers.
The main theorem (Theorem \ref{thm}) gives us the exact relation between them.
(Figure \ref{fig0}).
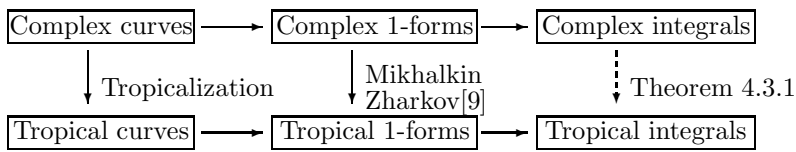
\begin{figure}[htbp]
\begin{center}
\begin{picture}(282,52)
\put(0,40){\framebox(70,12){Complex curves}}
\put(35,20){Tropicalization}
\put(0,0){\framebox(70,12){Tropical curves}}
\put(100,40){\framebox(77,12){Complex $1$-forms}}
\put(135,25){Mikhalkin}
\put(135,15){Zharkov\cite{Mik}}
\put(100,0){\framebox(77,12){Tropical $1$-forms}}
\put(200,40){\framebox(82,12){Complex integrals}}
\put(235,20){Theorem \ref{thm}}
\put(200,0){\framebox(82,12){Tropical integrals}}
\put(30,37){\vector(0,-1){20}}
\put(130,37){\vector(0,-1){20}}
\multiput(230,37)(0,-4){5}{\line(0,-1){2}}
\put(230,17){\vector(0,-1){0}}
\put(73,6){\vector(1,0){23}}
\put(73,46){\vector(1,0){23}}
\put(179,6){\vector(1,0){17}}
\put(179,46){\vector(1,0){17}}
\end{picture}
\end{center}
\caption{Theorem \ref{thm} shows us the direct relation between
two kinds of ``integrals".}
\label{fig0}
\end{figure}
\bigskip

{\bf Note :}
Throughout this paper, $\ep$ is a small real parameter $0<\ep<1$.
The symbol $\ee$ denotes the real number $e^{-1/\ep}$.
A \textit{formal Puiseux series with respect to} $\ee$ is a formal sum
of the form $\sum_{i=-n}^\infty{a_i\ee^{i/d}}$, where $a_i$ is a complex number
and $d$ is a positive integer.
If a formal Puiseux series converges for $\ep$ sufficiently small,
it is called \textit{convergent Puiseux series}.
$K$ denotes the field of convergent Puiseux series.
The field $K$ has the 
standard non-archimedean valuation $\val:K\to\QQ\cup\{+\infty\}$.
($\val(0)=+\infty$).
In this paper, 
we assume that all the Puiseux series considered are convergent,
unless otherwise stated.

\section{Approximation of hypersurfaces}\label{sec2}

\subsection{PL-polynomials and tropical hypersurfaces}

The purpose of this section is to give a brief review of the method 
for the \textit{approximation of hypersurfaces} 
of algebraic tori
given in \cite[\S6]{Viro2}.

Let $\CC$ be the complex number field and $\CC^n$ be the complex $n$-space.
Throughout this paper, $U$ denotes 
the unit circle $\{x\in\CC\,\vert\,\zet{x}=1\}$,
and $\CC\RR^n$ denotes 
the algebraic torus 
$\{(x_1,x_2,\dots,x_n)\,\vert\,x_1x_2\cdots x_n\neq 0\}$.

For a small positive parameter $0<\ep<1$,
define the maps $l(\ep):\CC\RR^n\to\RR^n$ and $a:\CC\RR^n\to U^n\
(:=U\times\dots\times U)$ by the formulae
\[
l(\ep)(x_1,\dots,x_n)=(-\ep\log{\zet{x_1}},\dots,-\ep\log{\zet{x_n}}),\quad
a(x_1,\dots,x_n)=\left(\frac{x_1}{\zet{x_1}},\dots,
\frac{x_n}{\zet{x_n}}\right).
\]

It is clear that the map
$la(\ep):\CC\RR^n\to\RR^n\times U^n$ defined by 
$x\mapsto (l(\ep)(x),\,a(x))$ is a diffeomorphism for any $\ep$.

For $w\in \RR$ and $\ep>0$, denote by $\mathcal{Q}_{w,\ep}$ the transformation
$\CC\RR^n\to\CC\RR^n$ defined by
\[
\mathcal{Q}_{w,\ep}(x_1,\dots,x_n)=(e^{-w_1/\ep}x_1,\dots,e^{-w_n/\ep}x_n),
\qquad
\mbox{ where }\
w=(w_1,\dots,w_n).
\]
We abbreviate the symbol $\mathcal{Q}_{w,\ep}$ as $\mathcal{Q}_w$ if
there is no confusion.

Let $T_w:\RR^n\to \RR^n$ be the translation $x\mapsto x+w$.
By definition,
we can derive the relation
\begin{equation}
la(\ep)\circ \mathcal{Q}_{w} \circ la(\ep)^{-1}=T_w\times \mathrm{id}_{U^n}.
\end{equation}
\bigskip

Our main object is an algebraic hypersurface of $\CC\RR^n$ defined
by a Laurent polynomial, which is an element of the ring
$\CC[x_1,x_1^{-1},\dots,x_n,x_n^{-1}]$.
Denote by $V_{\CC\RR^n}(f)$ the algebraic set in $\CC\RR^n$
defined by the Laurent polynomial 
$f(x_1,\dots,x_n)$, and 
let $V_{W}(f):=V_{\CC\RR^n}(f)\cap W$
for a subset $W\subset \CC\RR^n$.

For $w=(w_1,\dots,w_n)\in\ZZ^n$
and ordered $n$ variables $x=(x_1,\dots,x_n)$,
we abbreviate the monomial $x_1^{w_1}\cdots x_n^{w_n}$
as $x^w$.
Let $\{V_{\CC\RR^n}(f_\ep)\}_\ep$ be
an one-parameter family of algebraic hypersurfaces,
where $\ep$ is a positive real parameter and 
$f_\ep=f_\ep(x_1,\dots,x_n)$ is a Laurent polynomial
with coefficients depending on $\ep$.
In this paper, we mainly consider the polynomials
$f_\ep$ of the form:
\[
f_\ep(x)=\sum_{w\in\ZZ^n}{a_w(\ep)\,x^w},\qquad 
x=(x_1,\dots,x_n),\qquad \ a_w(\ep)\in K,
\]
where $a_w(\ep)\equiv 0$ except for finitely many $w\in\ZZ^n$. 
We call this type of polynomial a \textit{parameterised L-polynomial},
or \textit{pL-polynomial}.

Define
a \textit{tropical polynomial $\Val(X;f_\ep)$ 
associated with} $f_\ep$ by the formula
\[
\Val(X;f_\ep):=\min_{w\in\ZZ^n}[\val(a_w)+w_1X_1+\dots+w_nX_n],\qquad
X=(X_1,\dots,X_n).
\]

A \textit{tropical hypersurface defined by} $f_\ep$ is
a subset of $\RR^n$ defined by 
\begin{equation}\label{eq-x2}
\left.\left\{\phantom{\min_{w\in\ZZ^n}}\hspace{-0.7cm}
P=(A_1,\dots,A_n)\in \RR^n\,\right\vert\,
\mbox{the function }
\Val(X;f_\ep)
\mbox{ is not smooth at } X=P
\right\}.
\end{equation}
(For explicit examples for $n=2$, see Section \ref{sec3.1}).
We denote this tropical hypersurface by $TV_{\RR^n}(f_\ep)$.

Let $\mathfrak{a}=\{(X,\Val(X;f_\ep))\,\vert\,X\in\RR^n\}
\subset\RR^n\times\RR$ be the graph of 
$\Val(X;f_\ep)$. 
Clearly, $\mathfrak{a}$ is the skeleton of an $(n+1)$-dimensional convex 
(unbounded) polytope.
The tropical hypersurface $TV_{\RR^n}(f_\ep)$ 
is the image of the collection of 
$(n-1)$-faces in $\mathfrak{a}$
by the natural projection $\RR^{n}\times\RR\to\RR^n$.

\subsection{Canonical expressions of pL-polynomials}\label{sec2.2}

We
define $\snorm{f}:=\max_w{\zet{a_w}}$
for a Laurent polynomial $f(x_1,\dots,x_n)=\sum_w{a_wx^w}$.

Let $P=(A_1,\dots,A_n)$ be a point in $\RR^n$.
There exist finitely many 
integer vectors $w^{(1)},\dots,w^{(\alpha)}\in\ZZ^n$
such that the equations
$\Val(P;f_\ep)=\val(a_{w^{(i)}})+w^{(i)}_1A_1+\dots+w^{(i)}_nA_n$,
$(i=1,2,\dots,\alpha)$
hold.
Let $\Theta(P)$ be a set of these vectors.
Define the polynomial $\tilde{f}_\ep^P$ by the formula
$\tilde{f}_\ep^P=\displaystyle
\sum_{w\in\Theta(P)}{a_w(\ep)\,x^w}$.
\begin{rem}\label{rem-ex}
Because $\Theta(P)=\{w\}$ implies the formula $\Val(X;f_\ep)=
\val(a_w)+w_1X_1+\dots+w_nX_n$
around $X=P$,
the number of elements of $\Theta(P)$ is always greater than one for
$P\in TV_{\RR^n}{(f_\ep)}$.
\end{rem}

We note that the following two easy lemmas:
\begin{lemma}\label{lemma-ex-2}
Let $O=(0,\dots,0)$ be the origin of $\RR^n$. Then,
\[
\Val(P;f_\ep)=\Val(O;f_\ep\circ\mathcal{Q}_P),\qquad\mbox{where}\qquad
f_\ep\circ\mathcal{Q}_P(x)=f_\ep(\mathcal{Q}_P(x)).
\]
\end{lemma}
\begin{proof}
It is clear by the definition of $\Val(P;f_\ep)$.
\end{proof}
\begin{lemma}\label{lemma-ex-3}
$(\widetilde{f_\ep\circ\mathcal{Q}_P})^O=\tilde{f}_\ep^P\circ\mathcal{Q}_P$.
\end{lemma}
\begin{proof}
Let $f_\ep=\sum_w{a_w(\ep)x^w}$ and $P=(A_1,\dots,A_n)$. Then,
\begin{align*}
\textstyle(\widetilde{f_\ep\circ\mathcal{Q}_P})^O&=\textstyle
\sum_{\sharp}{a_w\ee^{A_1w_1+\dots+A_nw_n}x^w}
=\sum_{w\in\Theta(P)}{a_w\ee^{A_1w_1+\dots+A_nw_n}x^w}
=\tilde{f}_\ep^P\circ\mathcal{Q}_P,
\end{align*}
where $\sharp$ means \lq $\{w\,\vert\,
\val(a_w)+A_1w_1+\dots+A_nw_n=
\Val(O;f_\ep\circ\mathcal{Q}_P)$\}
\rq.
\end{proof}

Next we consider the decomposition:
\[\textstyle
f_\ep=
\left(\sum_{w\in\Theta(O)}+
\sum_{w\not\in\Theta(O)}
\right)
(a_w(\ep)\,x^w)
=\tilde{f}_\ep^O+\sum_{w\not\in\Theta(O)}{a_w(\ep)\,x^w}.
\]
By definition, the elements of the set $\Theta(O)$ satisfy the following relation:
$w\in\Theta(O),v\not\in\Theta(O)\Rightarrow 
\Val(O;f_\ep)=\val(a_{w})
<\val(a_{v})$. Therefore the pL-polynomial $\ee^{-\Val(O;f_\ep)}f_\ep$
can be decomposed as
$\ee^{-\Val(O;f_\ep)}f_\ep=f_1+f_2$, 
where 
\begin{gather*}
\textstyle
f_1=\ee^{-\Val(O;f_\ep)}\tilde{f}_\ep^O=\sum_w{b(\ep)\,x^w} \text{ s.t. } 
\val(b(\ep))=0,\\ 
\textstyle
\text{and }\qquad f_2=\sum_w{b'(\ep)\,x^w}  \text{ s.t. } \val(b'(\ep))>0.
\end{gather*}

Seeing the facts that 
i) $\val(a(\ep))=0\Leftrightarrow \lim_{\ep\to 0^+}{a(\ep)}\in \CC\setminus \{0\}$,
\ \ 
ii) $\val(a(\ep))>0\Leftrightarrow \lim_{\ep\to 0^+}{a(\ep)}=0$
for $a(\ep)\in K$, we can decompose uniquely
the pL-polynomial $\ee^{-\Val(O;f_\ep)}f_\ep$ as:
\begin{equation}\label{ex-3}
\ee^{-\Val(O;f_\ep)}f_\ep=f^O+\Delta(f_\ep),
\end{equation}
where $f^O$ is the Laurent polynomial \ $
\lim_{\ep\to 0^+}
{\ee^{-\Val(O;f_\ep)}\tilde{f}_\ep^O}$\
and
$\snorm{\Delta(f_\ep)}$
becomes to zero when $\ep\to 0^+$.
Of course, $f^O$ does not depend on $\ep$.

We can derive a more general formula soon:

\begin{prop}\label{ex-prop-2.1}
Let $f_\ep$ be a pL-polynomial and $P$ be a point in $\RR^n$.
Then the pL-polynomial $\ee^{-\Val(P;f_\ep)}f_\ep\circ \mathcal{Q}_P$
 is uniquely decomposed as:
\begin{equation}\label{ex-4}
\ee^{-\Val(P;f_\ep)}f_\ep\circ\mathcal{Q}_P=
{f}^P+\Delta(f_\ep\circ\mathcal{Q}_P),
\end{equation}
where $f^P$ is the Laurent polynomial $\lim_{\ep\to 0^+}{\ee^{-\Val(P;f_\ep)}
\tilde{f}_\ep^P\circ\mathcal{Q}_P}$ and 
$\snorm{\Delta(f_\ep\circ \mathcal{Q}_P)}$ becomes zero when $\ep\to 0^+$.
\end{prop}

\begin{proof}
To prove this, it is sufficient to substitute 
$f_\ep\mapsto f_\ep\circ\mathcal{Q}_P$ to
$(\ref{ex-3})$ and to use
Lemmas \ref{lemma-ex-2} and \ref{lemma-ex-3}.
\end{proof}

Hearafter we denote $\mathcal{R}^P(f_\ep):
=\ee^{-\Val(P;f_\ep)}f_\ep\circ\mathcal{Q}_P$
and $\Delta^P:=\Delta(f_\ep\circ\mathcal{Q}_P)$ for simplicity. Then 
(\ref{ex-4}) is expressed as $\mathcal{R}^P(f_\ep)=f^P+\Delta^P$.
We call this expression the \textit{canonical expression of $f_\ep$
at $P$}. The Laurent polynomial $f^P$ is considered as the \lq main part\rq\,
of $\mathcal{R}^{P}(f_\ep)$. We call it the \textit{$P$-truncation of $f_\ep$}.
Note that $\mathcal{R}^P(f_\ep)$ is continuous with respect to
$\ep$ and $P$, but $f^P$ and $\Delta^P$ are continuous with respect to
$\ep$ only.

\subsection{Approximation theorem (local version)}

Let $M$ be a smooth submanifold of a smooth manifold $X$.
A \textit{tubular neighbourhood} of $M$ in $X$ is a submanifold $N\subset X$
such that (i) $M\subset \mathrm{Int}N$, (ii) there exists a smooth retraction
$p:N\to M$ such that $p^{-1}(x)$ is diffeomorphic to the ($\dim X-\dim M$)-dimensional
ball for any $x\in M$.
If $X$ is equipped with a metric, a tubular neighbourhood $N$ of $M$ is called a
\textit{tubular $\mu$-neighbourhood} when any fibre $p^{-1}(x)$ is contained in a
ball of radius $\mu$ centred at $x$.

Now we introduce a flat metric into the space $\RR^n\times U^n$.
Let $d_{\RR^n}$ be the distance function over $\RR^n$
defined by Euclidean metric and $d_{U^n}$ be the distance function over $U^n$
defined by the standard flat metric of torus $U$.
Define the distance function 
`$\mathrm{dist}_{(\ep)}$' over $\RR^n\times U^n$
by the formula 
$\mathrm{dist}_{(\ep)}:=\ep^{-1}d_{\RR^n}+d_{U^n}$.
When a tubular neighbourhood contained in $\RR^n\times U^n$
is also a tubular $\mu$-neighbourhood with respect to
$\mathrm{dist}_{(\ep)}$, it is called a
\textit{tubular $(\mu,\ep)$-neighbourhood}.

\begin{rem}
Readers might suspect that it would be unnecessary to define such a
complicated 
distance $\mathrm{dist}_{(\ep)}$.
In fact, it is enough to consider the simpler distance function $d_{\RR^n}+d_{U^n}$
in order to prove the approximation theorem \ref{approximationtheorem}.
However, this is an essential procedure
for the method of approximation.
See Remark \ref{addrem}.
\end{rem}

The main role of tubular neighbourhoods is to formalise 
the approximation of hypersurfaces of $\RR^n\times U^n$.
For later arguments, it is convenient to consider some special class of
tubular neighbourhoods.
The tubular neighbourhood $p:N\to M$ is called \textit{normal}
if (i) any fibre $p^{-1}(x)$ consists of segments of geodesics,
(ii) any fibre $p^{-1}(x)$ intersects with $M$ orthogonally at $x$.
Note that two normal tubular neighbourhoods $p:N\to M$ and $p':N'\to M$
coincides with each other on $N\cap N'$.
\bigskip

Denote the subset 
$\{x=(x_1,\dots,x_n)\in\CC\RR^n\,\vert\,
1/r<\zet{x_i}< r,\ \forall i\}$ by $D(r)$ 
for a positive number $r>1$.
Let $\Lambda$ be a finite subset of $\ZZ^n$,
and
$f=\sum_{w\in\Lambda}{\alpha_w x^w}$ and
$g=\sum_{w\in\Lambda}{\beta_w x^w}$ 
be Laurent polynomials. Define $F:=f+g$.
We consider the behaviour of two algebraic sets $V_{\CC\RR^n}(F)$ and $V_{\CC\RR^n}(f)$
when $\snorm{g}$ goes to zero without changing $f$.
Assume that $V_{\CC\RR^n}(f)$ is a smooth hypersurface
of $\CC\RR^n$.
Standard arguments based on the Implicit Function Theorem give us the following
lemma:

\begin{lemma}\label{lemma-ex-4}
Fix $r>1$ arbitrarily.
Then,
for arbitrary $\mu_0>0$, there exists a positive number $\delta_0$ such that:
$$
\norm{g}<\delta_0
\ \Rightarrow\ 
\left(
\begin{array}{ll}
V_{D(r)}(f) \text{ is a smooth section of} \\
\text{a normal tubular $\mu_0$-neighbourhood }
N\to V_{D(r)}(F).
\end{array}
\right).
$$
\end{lemma}

We call $f_\ep$  
\textit{non-singular} if there exists a positive number $\delta>0$
such that $\ep\in(0,\delta]\Rightarrow V_{\CC\RR^n}(f_\ep)$
is non-singular,
and we call $f_\ep$ \textit{totally non-singular}
if, for all $P\in TV_{\RR^n}(f_\ep)$, $f^P$ is non-singular 
(with the usual meaning).
\begin{rem}
Totally non-singularity does not imply non-singularity.
\end{rem}

Recall that $la(\ep):\CC\RR^n\to \RR^n\times U^n$ is a diffeomorphism
between two topological spaces.
The following theorem is an essential part of the method of approximation.

\begin{theorem}[Approximation theorem at the origin]
\label{ex-thm-3.1}
Assume that $f_\ep$ is a non-singular and totally non-singular pL-polynomial, and
that the origin $O=(0,\!\dots,\!0)$ of $\RR^n$ is contained in 
$TV_{\RR^n}(f_\ep)$.
Let $\mathcal{R}^Of_\ep=f^O+\Delta^O$ be the canonical expression of $f_\ep$
at $O$.
Then
for arbitrary $\mu>0$, there exists a positive number $\delta>0$ 
such that
\begin{align*}
&\norm{\Delta^O}<\delta\Rightarrow\\
&\left(
\begin{array}{l}
\mbox{There exists an open neighbourhood $W_\ep$ of $O\in\RR^n$ such that}\\
\{W_\ep\times U^n\} \bigcap \left\{la(\ep)(V_{\CC\RR^n}(f^O))\right\}
\mbox{ is a smooth 
section of}\\
\mbox{ a tubular $(\mu,\ep)$-neighbourhood }
N\to \{W_\ep\times U^n\} \bigcap 
\left\{la(\ep)(V_{\CC\RR^n}(f_\ep))\right\}.
\end{array}
\right).
\end{align*}
Moreover, we can assume that 
the preimage of this tubular neighbourhood by $la(\ep)$ is normal.
\end{theorem}
\begin{proof}
Fix a positive number $r>1$ arbitrarily and let $\mu_0:=(2\sqrt{n}r)^{-1}\!\cdot\mu$.
To begin with, 
note that we can derive soon the following relation by direct calculations:
\[
\mathrm{dist}_{\CC\RR^n}(\alpha,\beta)<\mu_0,\ \
\alpha,\beta\in D(r)\quad\Rightarrow\quad 
\mathrm{dist}_{(\ep)}(la(\ep)(\alpha),la(\ep)(\beta))<\mu.
\]
By Lemma \ref{lemma-ex-4}, 
there exists a small number $\delta>0$
such that \lq$\norm{\Delta^O}<\delta\Rightarrow 
V_{D(r)}(f^O)$
is a smooth section of a normal tubular $\mu_0$-neighbourhood 
$\mathcal{N}\to V_{D(r)}(f_\ep)$
$(=V_{D(r)}(\mathcal{R}^Of_\ep))$\rq.

Therefore, it is sufficient to define $W_\ep:=l(\ep)(D(r))$ and
$N:=la(\ep)(\mathcal{N})$.
Clearly, we have $W_\ep\ni O$.
\end{proof}
The slight extension of 
Theorem \ref{ex-thm-3.1} can be proved:
\begin{cor}[Approximation theorem (local version)]
\label{ex-cor-3.1}
Assume $f_\ep$ is a pL-polynomial satisfying the same condition stated in 
Theorem \ref{ex-thm-3.1}.
Let $P$ be a point in $TV_{\RR^n}(f_\ep)$ and $\mathcal{R}^Pf_\ep=f^P+\Delta^P$
be the canonical expression of $f_\ep$ at $P$.
Then
for arbitrary $\mu>0$, there exists a positive number $\delta>0$ 
such that
\begin{align*}
&\norm{\Delta^P}<\delta\Rightarrow\\
&\left(
\begin{array}{l}
\mbox{There exists an open neighbourhood $W_\ep$ of $P\in\RR^n$ such that}\\
\{W_\ep\times U^n\} \bigcap 
\left\{(T_P\times \mathrm{id}_{U^n})\circ
la(\ep)(V_{\CC\RR^n}(f^P))\right\}
\mbox{ is a smooth 
section of}\\
\mbox{ a tubular $(\mu,\ep)$-neighbourhood }
N\to \{W_\ep\times U^n\} \bigcap \left\{la(\ep)(V_{\CC\RR^n}(f_\ep))\right\}.
\end{array}
\right).
\end{align*}
Moreover, if $f^{P}=f^Q$ ($P,Q\in TV_{\RR^n}(f_\ep)$), we can take same $\delta$
for these two points.
\end{cor}
\begin{proof}
Let $\tilde{T}:=T_P\times id_{U^n}$.
To prove the first statement, 
it is sufficient to
to apply Theorem \ref{ex-thm-3.1} to the canonical expression 
$\mathcal{R}^P(f_\ep)=f^P+\Delta^P$ and
to use the following equation:
\begin{gather*}
\tilde{T}\circ la(\ep)(V(\mathcal{R}^Pf_\ep))\!=\!
\tilde{T}\circ la(\ep)(V(f_\ep\circ \mathcal{Q}_P))\!=\!
\tilde{T}\circ la(\ep)\circ \mathcal{Q}_P^{-1}(V(f_\ep))=la(\ep)(V(f_\ep)).
\end{gather*}
The second statement follows from the fact that
$\tilde{T}$ does not change the metric.
\end{proof}

\subsection{Approximation theorem (global version)}

Theorem \ref{ex-thm-3.1} and Corollary \ref{ex-cor-3.1}
show us the existence of a small region
in which two varieties $V_{\CC\RR^n}(f_\ep)$ and $V_{\CC\RR^n}(f^P)$
are \lq similar\rq.
We extend this region by \textit{gluing} tubular neighbourhoods. 

As mentioned above, the tropical hypersurface $TV_{\RR^n}(f_\ep)$ is an union of
finitely many
(not necessarily bounded)
$(n-1)$-faces. This tropical hypersurface
is an image of $(n-1)$-faces of a convex $(n+1)$-polytope by the projection
$\RR^n\times\RR\to\RR^n$. Therefore,
naturally $TV_{\RR^n}(f_\ep)$ has a cell decomposition.
We express this decomposition as $TV_{\RR^n}(f_\ep)=\bigcup_{\lam\in\Lam}
{\mathcal{X}_\lam}$ formally.
The index set $\Lam$ is finite.

\begin{lemma}\label{ex-lemma-0.4}
For $P_1,P_2\in\mathcal{X}_\lam$, 
the truncations $f^{P_1}$ and $f^{P_2}$ coincide with each other.
\end{lemma}
\begin{proof}
Let $\mathfrak{a}\in\RR^n\times\RR$ be the graph of $\Val(X;f_\ep)$ and 
$p:\mathfrak{a}\stackrel{\sim}{\to}\RR^n$ be a restriction of the projection
$\RR^n\times\RR\to\RR^n$ to $\mathfrak{a}$.

Assume $f^{P_1}\neq f^{P_2}$. This means $\Theta(P_1)\neq \Theta(P_2)$.
We can assume $w=(w_1,\dots,w_n)$ $\in \Theta(P_1)$ and $w\not\in \Theta(P_2)$.
It follows that $p^{-1}(P_1)$ is contained in the face 
$\mathfrak{a}\cap
\{\Val(X;f_\ep)=\val(a_w)+w_1X_1+\dots+w_nX_n\}$, and that $p^{-1}(P_2)$ is not.
Then $P_1$ and $P_2$ are not contained in a same cell.
\end{proof}

We say that a smooth submanifold $V_1\subset X$ is 
$(\mu,\ep)$-\textit{approximated by} 
another smooth submanifold
$V_2\subset X$ \textit{around} a point $P\in X$
if there exists an open neighbourhood $W\ni P$ such that
$W\cap V_2$ is a section of a tubular $(\mu,\ep)$-neighbourhood 
$N\to W\cap V_1$.
By Corollary \ref{ex-cor-3.1} and 
Lemma \ref{ex-lemma-0.4}, there exists positive $\delta$ 
(that does not depend on $P$ because $\Lam$ is finite!)
such that
\[
\norm{\Delta^P}<\delta\ \Rightarrow\
\begin{array}{r}
la(\ep)(V(f_\ep))
\text{ is $(\mu,\ep)$-approximated by }
(T_P\times\mathrm{id})\circ la(\ep)(V(f^P))\\ 
\!\text{ around } P.
\end{array}
\]

\begin{lemma}
There exists a positive number $\zeta$ such that\ 
$\ep\in(0,\zeta]\Rightarrow\norm{\Delta^P}\!<\delta,\ \forall P\in 
TV_{\RR^n}(f_\ep)$.
\end{lemma}
\begin{proof}
Fix a small $0<\ep<1$. (Then $\ee=e^{-1/\ep}<1$).
It is sufficient to prove that 
$\left\{
\norm{\Delta^P}\,\left\vert\,P\in\mathcal{X}_\lam\right.
\right\}$ 
has an upper bound for each $\mathcal{X}_\lam$.
(Note that $\Lam$ is finite).
Clearly, $\norm{\Delta^P}$ is continuous with respect to
$P\in\mathcal{X}_\lam$.

(i) When the closure of $\mathcal{X}_\lam$ is compact,
it is sufficient to confirm the fact that 
$\lim_{P\to\partial\mathcal{X}_\lam}\norm{\Delta^P}$ is finite.

(ii) When $\mathcal{X}_\lam$ is unbounded,
we should consider the behaviour of $\norm{\Delta^P}$ when $\zet{P}\to\infty$.
The pL-polynomial $\Delta^P$ is of the form
$$\textstyle
\sum'_{w}{\ee^{\val(a_w)+w_1P_1+\dots+w_nP_n-\Val(P;f_\ep)}x^w},
$$
where $w$ runs over all the element of $\ZZ^n$ such that
$\val(a_w)+w_1P_1+\dots+w_nP_n-\Val(P;f_\ep)>0$.
Therefore, when $P$ goes infinity along $\mathcal{X}_\lam$,
the function $\val(a_w)+w_1P_1+\dots+w_nP_n-\Val(P;f_\ep)$ should 
grow to infinity or be constant.
Then the limit $\lim_{\zet{P}\to\infty,P\in\mathcal{X}_\lam}
\norm{\Delta^P}$ should be finite.
\end{proof}
\bigskip

Recall that two normal tubular neighbourhoods (or their images by $la(\ep)$)
coincides with each other on their intersection.
By gluing these local tubular neighbourhoods, we obtain a global tubular neighbourhood
which gives us the approximation theorem:

\begin{theorem}[Approximation theorem (global version)]
\label{approximationtheorem}
Let $f_\ep$ be a non-singular and totally non-singular pL-polynomial. 
Denote the canonical expression of $f_\ep$ at $P$ by
$\mathcal{R}^Pf_\ep=f^P+\Delta^P$.
Then
for arbitrary $\mu>0$, there exists a positive number $\zeta>0$ 
such that
\begin{align*}
&\ep\in(0,\zeta]\Rightarrow\\
&\left(
\begin{array}{l}
\mbox{There exists a tubular $(\mu,\ep)$-neighbourhood 
$N\to la(\ep)(V_{\CC\RR^n}(f_\ep))$ 
such that}\\
(T_P\times \mathrm{id}_{U^n})\circ la(\ep)(V_{\CC\RR^n}(f^P))\mbox{ is a smooth 
section of it}
\mbox{ around $P$.}
\end{array}
\right).
\end{align*}
\end{theorem}

\begin{rem}\label{addrem}
For two distinct points $\alpha,\beta$ of $V(f^P)$, the distance 
$$
\mathrm{dist}_{(\ep)}(la(\ep)(\alpha),la(\ep)(\beta))
$$
does not depend on $\ep$. (Recall the definition of $la(\ep)$).
This fact is the reason we defined the seemingly complicated distance 
$\mathrm{dist}_{(\ep)}$.
If this distance were not independent of $\ep$,
Theorem \ref{approximationtheorem} would not
give us any approximation of hypersurfaces.
\end{rem}

\subsection{Surjectivity theorem}

In this section, we introduce the \textit{surjectivity theorem} without proof. 
For details, the reader should consult the following references: 
Einsiedler, M. Kapranov, M. and Lind, D. \cite{Lind};
Payne, S. \cite{Payne}.

Let $f_\ep(x)$ be a pL-polynomial in $n$ valuables $x_1,\dots,x_n$ and
$\mathcal{R}^Pf_\ep=f^P+\Delta^P$ be its canonical expression at a point 
$P=(A_1,\dots,A_n)$
in $TV_{\RR^n}(f_\ep)$. Then we have the following theorem.

\begin{theorem}[Surjectivity theorem] \label{surjectivitytheorem}
Let $p=(p_1,\dots,p_n)\in V_{\CC\RR^n}(f^P)$.
Then there exist $n$ Puiseux series 
$\tilde{p}_1=p_1\ee^{A_1}+p_1'\ee^{A_1'}+p''_1\ee^{A_1''}+\cdots$,
$\cdots$, $\tilde{p}_n=p_n\ee^{A_n}+p_n'\ee^{A_n'}+p''_n\ee^{A_n''}+\cdots$
such that 
$(\tilde{p}_1,\dots,\tilde{p}_n)\in V_{\CC\RR^n}(f_\ep)$.
\end{theorem}

This theorem states information about pointwise convergence
of $V_{\CC\RR^n}(f_\ep)$. 
Briefly, the approximation theorem \ref{approximationtheorem} deals with
global information of $V_{\CC\RR^n}(f_\ep)$
and the surjectivity theorem \ref{surjectivitytheorem}
deals with local information.

\section{Plane Curves over $K$ and Tropical Curves}\label{sec3}

In the rest of this paper, 
we consider the two-dimensional case. In this case the varieties 
$V_{\CC\RR^2}(f_\ep)$ and $V_{\CC\RR^2}(f^P)$ are complex curves 
(or Riemannian surfaces)
contained in the 
algebraic torus $\CC\RR^2$. 
We assume $V_{\CC\RR^2}(f_\ep)$ is non-singular and totally non-singular unless
otherwise stated. Now we denote $V_{\CC\RR^2}(f_\ep)$, $V_{\CC\RR^2}(f^P)$,
$TV_{\RR^2}(f_\ep)$, \dots\textit{etc.} by $V(f_\ep)$, $V(f^P)$, $TV(f_\ep)$,
\dots\textit{etc.} for simplicity.

Recall that $K$ is the Puiseux series field
with the standard non-archimedean valuation $\val$.
Define the multiplicative group $R^\times:=\{x\in K\,\vert\,\val(x)=0\}$.
Any element $x$ of $R^\times$ tends to a finite non-zero complex number
when $\ep$ tends to zero. Denote the limit by $\htht\,{(x)}$.
Let 
\begin{equation}\label{eq3.1}
\textstyle f_\ep(x,y)=\sum_{i=0}^N{a_i(x)y^{N-i}}\,
=a_0(x)\,y^N\!\!+a_1(x)\,y^{N-1}+\!\cdots\!+a_N(x)
\end{equation}
be a polynomial over $K$.
As $K$ is algebraically closed, $a_i$  $(i=0,1,\dots,N)$ has the expression:
\begin{equation}\label{eq3.2}
\textstyle
a_i(x)=c_i\ee^{A_i}x^{m_i}\prod_{j=1}^{d_i}{(x-u_{i,j}\ee^{B_{i,j}})},
\quad
c_i,u_{i,j}\in R^\times\!,\ A_i,B_{i,j}\in\QQ,\,m_i\in\NN.
\end{equation}
Define the algebraic curve $C_\ep:=V(f_\ep)$.
By changing variables $x\mapsto x\ee^{-R}$ and
$y\mapsto y\ee^{-R'}$ $(R,R'\!>\!\!>\!0)$,
we can assume $A_i,B_{i,j}>0$
without loss of generality.
In the present paper, we investigate algebraic curves over $K$
which satisfy
the following genericness condition:\\

\textbf{Genericness condition.}\ \ 
The numbers
$\htht (u_{i,j})$ $\forall{i,j}$ are all distinct.\\

For the proof of our main theorem \ref{thm},
we will impose a slightly stricter conditions on the curve.
We discuss these conditions in the appendix.

\subsection{Examples}\label{sec3.1}
We first give some examples of plane curves over $K$
and their tropicalization.
These examples show us how to approximately reconstruct the plane curve
from its tropicalization.
\subsubsection{Example (I)}
Let $C_\ep$ be the curve defined by the pL-polynomial
\begin{equation}\label{eq3.3b}
f_\ep(x,y)=(x+\ee)y^2+(x+\ee^2)(x+\ee^3)y+\ee^8=0.
\end{equation}
The tropicalization of $C_\ep$
is 
\begin{equation}
\begin{array}{ll}
TV(f_\ep)\!\!&:=
\left
\{(X,Y)\in\RR^2
\left\vert
\begin{array}{l}
\min{[X+2Y,2Y+1,2X+Y,X+Y+2,Y+5,8]}\\
\mbox{is not smooth.}
\end{array}
\!\right\}\right..
\end{array}
\end{equation}
Let us denote this variety by $\mathrm{Trop}\,C$ simply.
Figure \ref{fig2}
shows $\mathrm{Trop}\,C$.
$\mathrm{Trop}\,C$ has four vertices $\alpha=(1,1), \beta=(2,4), \gamma=(2,3)$ and 
$\delta=(2.5,3.5)$, and has one closed loop.
We use the term ``\textit{edge}" only when it means a segment
of finite length. 
Edges of infinite length shall be called \textit{leaves}
hereafter.
The \textit{genus of} $\mathrm{Trop}\,C$ is the number of independent
closed cycles over $\mathrm{Trop}\,C$.
In this case, $\mathrm{genus}(\mathrm{Trop}\,C)=1$.
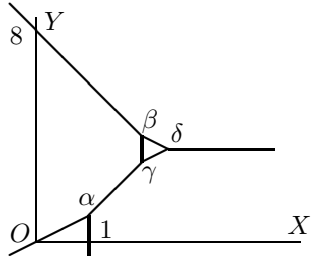
\begin{figure}[htbp]
\begin{center}
\begin{picture}(100,90)
\thinlines
\put(0,0){\line(1,0){100}}
\put(0,0){\line(0,1){85}}
\thicklines
\put(0,0){\line(2,1){20}}
\put(0,0){\line(-2,-1){10}}
\put(20,10){\line(1,1){20}}
\put(40,30){\line(2,1){10}}
\put(50,35){\line(1,0){40}}
\put(20,10){\line(0,-1){15}}
\put(0,80){\line(1,-1){40}}
\put(0,80){\line(-1,1){10}}
\put(40,30){\line(0,1){10}}
\put(40,40){\line(2,-1){10}}
\put(-10,0){$O$}
\put(3,80){$Y$}
\put(95,3){$X$}
\put(-10,75){$8$}
\put(24,1.5){$1$}
\put(16,14){$\alpha$}
\put(40,43){$\beta$}
\put(40,24){$\gamma$}
\put(51,38){$\delta$}
\end{picture}
\end{center}
\caption{The variety $\mathrm{Trop}\,C$.
It consists of four vertices, four edges and four leaves.
The genus of $\mathrm{Trop}\,C$ is one.}
\label{fig2}
\end{figure}

The canonical expression of $f_\ep$ at $\alpha$ is 
$\mathcal{R}^\alpha f_\ep=(xy^2+y^2+x^2y)+(\ee xy+\ee^2 y+\ee^3 y+\ee^5)$.
By the approximation theorem \ref{approximationtheorem}, 
the curve $la(\ep)(C_\ep)$ 
is approximated by the translation of 
$la(\ep)(V(f^\alpha))=la(\ep)(V(xy^2+y^2+x^2y))=la(\ep)(V_{\CC\RR^2}(xy+y+x^2))$ 
around $\alpha$.
Similarly, around the points $\beta,\gamma,\delta\in\mathrm{Trop}\,C$,
$la(\ep)(C_\ep)$ is approximated by
$la(\ep)(V(x^2y+xy+1))$, $la(\ep)(V(y+x^2+x))$ and $la(\ep)(V(y^2+xy+1))$ respectively.

Figure \ref{fig1}
shows four `local' Riemannian surfaces $V(f^\alpha),V(f^\beta),V(f^\gamma)$ and
$V(f^\delta)$.
\begin{figure}[htbp]
\begin{center}
\includegraphics[bb=39 624 526 797,clip, width=9cm]
{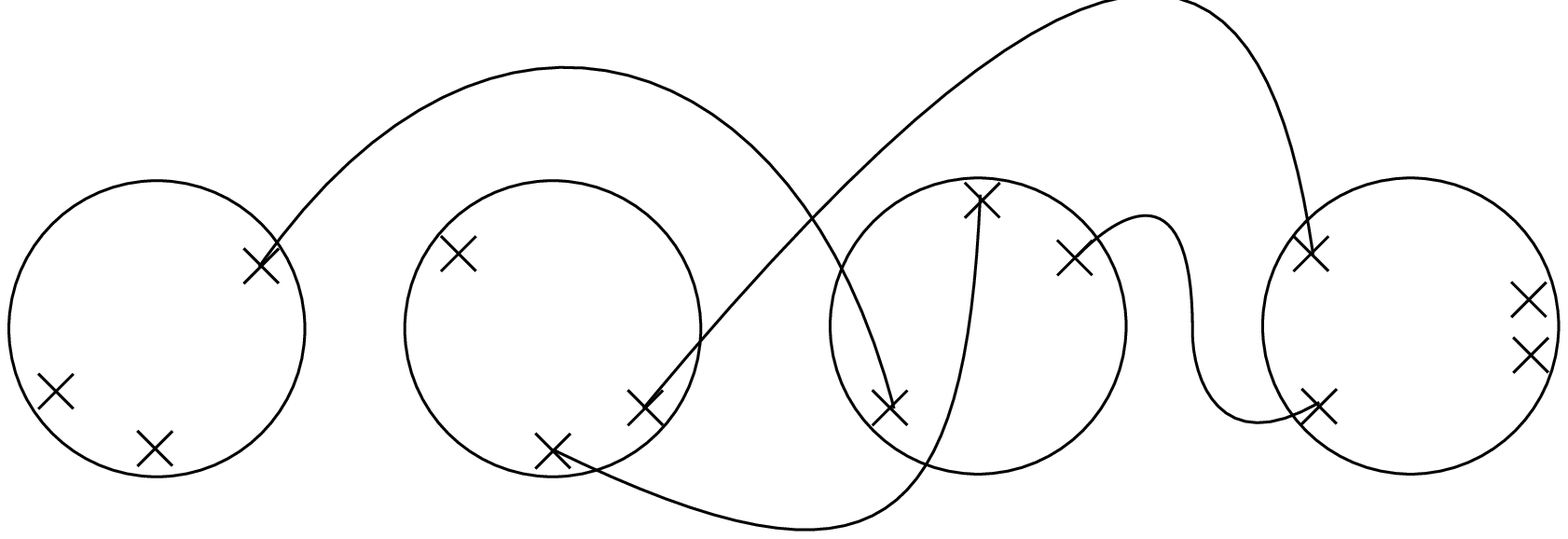}
\begin{picture}(300,-10)(0,-10)
\put(50,63){$(0,0)$}
\put(-10,22){$(\infty,\infty)$}
\put(30,5){$(-1,\infty)$}
\put(108,34){$(0,\infty)$}
\put(80,60){$(\infty,0)$}
\put(100,0){$(-1,\infty)$}
\put(190,60){$(0,0)$}
\put(135,15){$(\infty,\infty)$}
\put(170,70){$(-1,0)$}
\put(275,42){$(0,1)$}
\put(275,30){$(0,-1)$}
\put(215,9){$(\infty,\infty)$}
\put(235,65){$(\infty,0)$}
\put(277,0){${}^\infty\!\!\!\leftarrow x\rightarrow^0$}
\put(312,35){$y$}
\put(312,25){$\downarrow_\infty$}
\put(312,45){$\uparrow^0$}
\put(50,-10){$\alpha$}
\put(110,-10){$\beta$}
\put(180,-10){$\gamma$}
\put(255,-10){$\delta$}
\end{picture}
\end{center}
\caption{`Local' Riemannian surfaces.
All of genus $0$.
The points $(x,y)$ satisfying $x=\infty,0$ or $y=\infty,0$ are described
in the figure.
}
\label{fig1}
\end{figure}

According to the approximation theorem,
we can approximately draw the form of the 
Riemannian surface $C_\ep$ for small $\ep>0$.
Gluing the local data in Figure \ref{fig1} along $\mathrm{Trop}\,C$ 
(Figure \ref{fig2}),
we can sketch $la(\ep)(C_\ep)$ (or $C_\ep$ which is diffeomorphic to $la(\ep)(C_\ep)$) 
as in Figure \ref{fig3a}:
Four small spheres, four long cylinders and five horns
make up the figure.
For later arguments, we take the completion of $C_\ep$ that is a compact Riemannian
surface. 
\begin{figure}[htbp]
\begin{center}
\begin{picture}(0,0)
\put(80,60){$\alpha$}
\put(155,175){$\beta$}
\put(160,100){$\gamma$}
\put(225,140){$\delta$}
\end{picture}
\includegraphics[bb=0 383 565 810, clip,
width=10cm
]{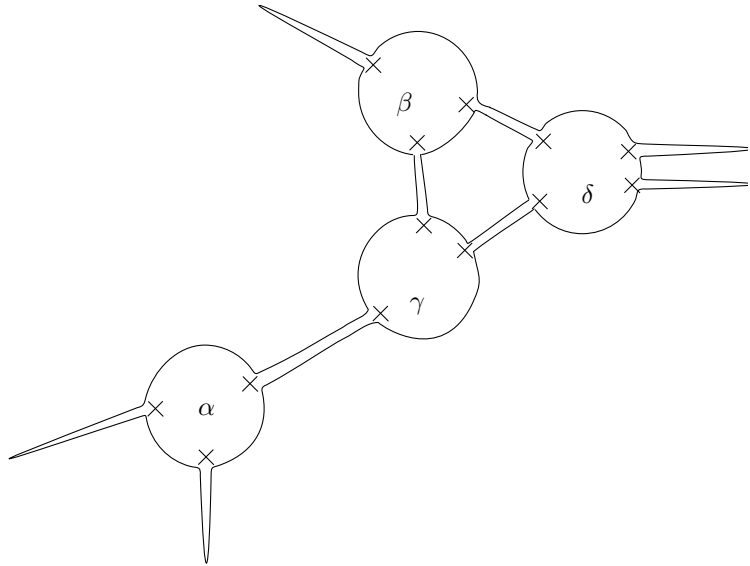}
\end{center}
\caption{A sketch of (the completion of) $C_\ep$
consisting of
four spheres, four cylinders and five horns.
}
\label{fig3a}
\end{figure}

\subsubsection{Example (II)}
Let us consider
\begin{equation}\label{eq3.13b}
C_\ep:y^3+(x+\ee^4)y^2+\ee^2(x+\ee)(x+2\ee)y+\ee^{10}=0
\end{equation}
and its tropicalization:
\[
\mathrm{Trop}\,C:
\min{
\left[
\begin{array}{l}
3Y,X+2Y,2Y+4,2X+Y+2,\\
X+Y+3,Y+4,10
\end{array}
\right]}
\mbox{ is not smooth.}
\]
$\mathrm{Trop}\,C$ has three vertices $\alpha=(1,6)$,
$\beta=(1,3)$ and $\gamma=(2,2)$.
The truncations associated with these points are:
$f^\alpha=(x+1)(x+2)y+1$, $f^\beta=xy^2+(x+1)(x+2)$ and $f^\gamma=y^3+xy^2+2y$ 
respectively.
Figure \ref{fig3} displays an approximate sketch of (the completion of) $C_\ep$.
Although $C_\ep$ is of genus one,
the genus of $\mathrm{Trop}\,C$ is zero. This difference comes from
the edge which connects $\alpha$ and $\beta$ in $\mathrm{Trop}\,C$.
Two long cylinders are associated with this one edge.

\begin{figure}[htbp]
\begin{center}
\begin{picture}(100,150)(0,-60)
\put(0,0){\line(1,0){100}}
\put(0,0){\line(0,1){90}}
\thicklines
\put(0,0){\line(2,1){40}}
\put(0,0){\line(-2,-1){10}}
\put(40,20){\line(1,0){60}}
\put(0,20){\line(2,1){20}}
\put(0,20){\line(-2,-1){10}}
\put(20,30){\line(2,-1){20}}
\put(20,30){\line(0,1){30}}
\put(20,60){\line(-1,1){30}}
\put(20,60){\line(1,0){80}}
\thinlines
\multiput(0,30)(4,0){5}{\line(1,0){2}}
\multiput(0,60)(4,0){5}{\line(1,0){2}}
\multiput(20,0)(0,4){10}{\line(0,1){2}}
\put(-10,0){$O$}
\put(95,3){$X$}
\put(3,85){$Y$}
\put(23,2){$1$}
\put(-10,75){$8$}
\put(-10,20){$2$}
\put(-10,29){$3$}
\put(-10,57){$6$}
\put(20,65){$\alpha$}
\put(25,30){$\beta$}
\put(38,12){$\gamma$}
\put(163,97){$\alpha$}
\put(163,14){$\beta$}
\put(240,-20){$\gamma$}
\end{picture}
\includegraphics[bb=71 432 438 786, clip,
width=7cm
]{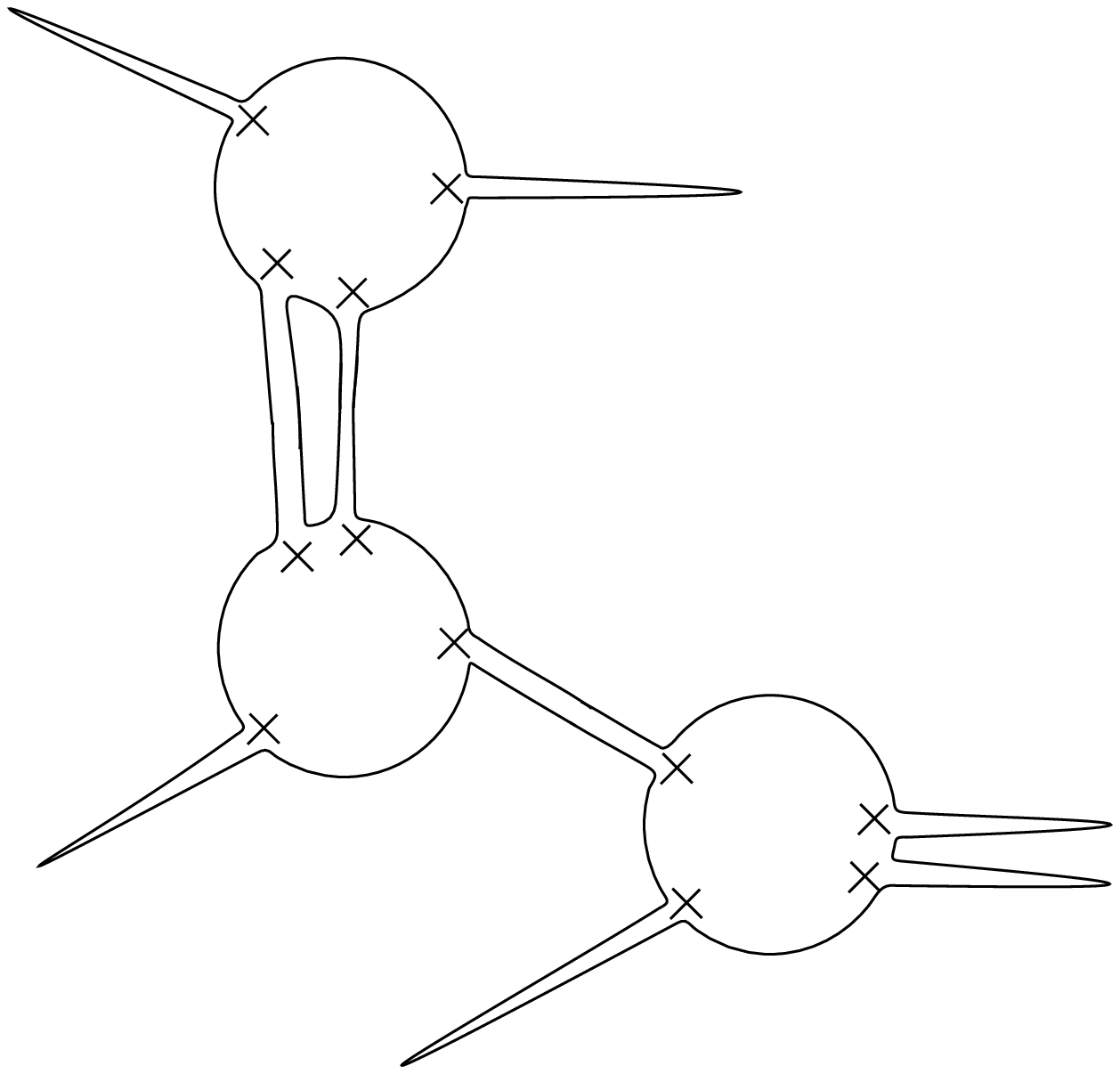}
\end{center}
\begin{center}
\includegraphics[bb=90 619 506 773, clip,
width=9cm
]{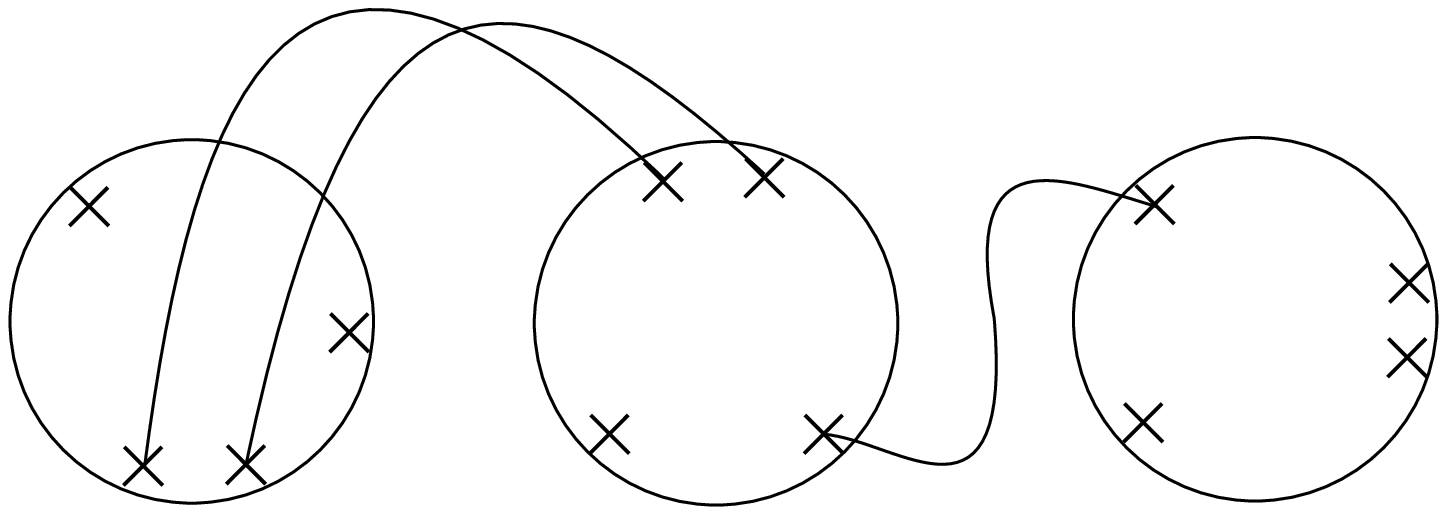}
\begin{picture}(0,0)(200,0)
\put(4,22){$(0,-\frac{1}{2})$}
\put(-61,70){$(\infty,0)$}
\put(-65,-5){$(-1,\infty)$}
\put(-20,-5){$(-2,\infty)$}
\put(38,45){$(-1,0)$}
\put(73,70){$(-2,0)$}
\put(27,-3){$(\infty,\infty)$}
\put(77,-3){$(0,\infty)$}
\put(125,70){$(\infty,0)$}
\put(125,-3){$(\infty,\infty)$}
\put(195,40){$(0,\ai\sqrt{2})$}
\put(195,27){$(0,-\ai\sqrt{2})$}
\put(210,0){${}^\infty\!\!\!\leftarrow x\rightarrow^0$}
\put(245,35){$y$}
\put(245,25){$\downarrow_\infty$}
\put(245,45){$\uparrow^0$}
\put(-50,27){$\alpha$}
\put(63,20){$\beta$}
\put(160,25){$\gamma$}
\end{picture}
\end{center}
\caption{The sketch of $C_\ep$.}
\label{fig3}
\end{figure}

\bigskip
In both examples above, the shapes of Riemannian surfaces
$V(f^P)$ ($P$ is a vertex of $\mathrm{Trop}\,C$)
are essential for drawing $C_\ep$.
We denote these Riemannian surfaces by 
\textit{Riemannian sub-surfaces}.

In example (II), there exists an edge associated to two cylinders.
This property reflects the fact that:
For $P=(1,Y)\in\mathrm{Trop}\,C\quad (3<Y<6)$,
the variety $V(f^P)=V((x+1)(x+2))=V(x+1)\coprod V(x+2)$ consists of 
two irreducible components. 
In such a case,
we call the edge $\overline{\alpha\beta}$ \textit{of multiplicity} $2$.

Here we remark on the behaviour of sub-surfaces.
It may happen that a sub-surface is
reducible or of genus more than $0$.
In these cases, the genus of $\mathrm{Trop}\,C$
becomes inferior to the genus of $C_\ep$.

The following definition is given for Section \ref{sec4}.

\begin{defi}
A curve over $K$ is called
non-degenerate if all of its 
Riemannian sub-surfaces
are irreducible and of genus $0$.
\end{defi}

\subsection{Multiplicity of tropical edges}

In this section, we define the 
\textit{vertical thickness}, the \textit{horizontal thickness} 
and the \textit{multiplicity}
of tropical edges.

\subsubsection{Vertical and horizontal thickness}
The `Val' function associated with the pL-polynomial $a_i(x)$ (\ref{eq3.2})
is expressed as
\begin{equation}\label{eq3.19}
\textstyle
\Val(X;a_i)=A_i+m_iX+\sum_{j=1}^{d_i}\min{[X,B_{i,j}]},\quad
(i=0,1,\dots,N).
\end{equation}
Let $F_i(X):=\Val(X;f_\ep)$.
Then the defining condition of $\mathrm{Trop}\,C$ is expressed as
\begin{equation}\label{eq3.20}
\mathrm{Trop}\,C:\ \ 
\mathrm{min}_{i=0,\dots,N}{[F_i(X)+(N-i)Y]}\quad \mbox{ is not smooth}.
\end{equation}
Of course, we have $\Val(X,Y;f_\ep)=\mathrm{min}_{i=0,\dots,N}{[F_i(X)+(N-i)Y]}$.
Define the domain $\mathfrak{D}_i\subset \RR^2$ $(i=0,1,\dots,N)$ by
\[
\mathfrak{D}_i:=\left\{(X,Y)\,\vert\,\Val(X,Y;f_\ep)
=F_i(X)+(N-i)Y\right\}.
\]
The domains $\mathfrak{D}_i$ separate the plane $\RR^2$ into at most $N+1$
pieces:
$\RR^2=\bigcup_{i=0}^N{\mathfrak{D}_i}$.
If $i<j$, $(x,y_i)\in\mathfrak{D}_i$ and $(x,y_j)\in\mathfrak{D}_j$ then
$y_i\leq y_j$.
($\because$ From the
definition of $\mathfrak{D}_i$ and $\mathfrak{D}_j$,
we have $F_i(x)+(N-i)y_i\leq F_j(x)+(N-j)y_i$ and
$F_j(x)+(N-j)y_j\leq F_i(x)+(N-i)y_j$.
It follows that $(j-i)y_i\leq (j-i)y_j$.)
Let us define the piecewise linear function $\mathcal{N}_i(X)$ $(i=1,2,\dots,N)$
defined by the relation
$
\mathcal{N}_i(X)=\min_{j\geq i}\{Y\,\vert\,(X,Y)\in\mathfrak{D}_{i} \}
$.
Note that $\mathcal{N}_{i+1}(X)\geq\mathcal{N}_i(X)$ for all $X$.

Using the function $\mathcal{N}_i(X)$, we obtain an another expression of 
$\mathrm{Trop}\,C$. We formally regard $\mathcal{N}_{N+1}(X):=+\infty$ 
and $\mathcal{N}_{0}(X):=-\infty$
for any $X$.
\begin{prop}\label{prop3.3}
Let $L_{i,j}$ be the vertical edge which connects 
\[(B_{i,j},\mathcal{N}_i(B_{i,j}))
\qquad\mbox{ and }\qquad (B_{i,j},\mathcal{N}_{i+1}(B_{i,j})).\]
Then, the set $\left(\bigcup_{i=1}^N{\{Y=\mathcal{N}_i(X)\}}\right)\cup
\left(\bigcup_{i,j}{L_{i,j}}\right)$
coincides with $\mathrm{Trop}\,C$.
\end{prop}
\begin{rem}
If $\mathcal{N}_{i}(B_{i,j})=\mathcal{N}_{i+1}(B_{i,j})$,
then
$L_{i,j}=\{\mbox{a point}\}$.
\end{rem}
\begin{proof}
Let $G_i:=\{(X,Y)\in\RR^2\,\vert\,Y=\mathcal{N}_i(X)\}$.
Because it is obvious, by definition,
that 
$\textstyle
\mathrm{Trop}\,C\supset
\left(\bigcup_{i=1}^N{G_i}\right)
$, it is sufficient to prove
$\textstyle
\mathrm{Trop}\,C\setminus
\left(\bigcup_{i=1}^N{G_i}\right)=\left(\bigcup_{i,j}{L_{i,j}}\right)$.
Choose a connected component $O$ of 
$\RR^2\setminus\left(\bigcup_{i=1}^N{G_i}\right)$.
The domain $O$ is contained in $\mathfrak{D}^\circ_i$ for some $i$,
where $\mathfrak{D}^\circ_i$ is the set of inner points of $\mathfrak{D}_i$.
For any point $(X,Y)$ in $O$, it follows that
$\Val(X,Y;f_\ep)=F_i(X)+(N-i)Y$.
Then, a point $(X,Y)\in\mathrm{Trop}\,C\cap O$ must be a point at which
$F_i(X)$ is not smooth.
Because the function $F_i(X)$ (\ref{eq3.19}) is not smooth 
iff
$X=B_{i,j}$, 
we conclude that $\mathrm{Trop}\,C\setminus
\left(\bigcup_{i=1}^N{G_i}\right)$
consists of
the sets $\{X=B_{i,j}\}\cap O=L_{i,j}$.
\end{proof}
\bigskip

By use of Proposition \ref{prop3.3},
we introduce 
the vertical thickness
of edges.
\begin{defi}\label{def3.3.1}
Let $E\subset\mathrm{Trop}\,C$ be an edge.
We call the number
$
\sharp\left\{i\,\vert\,E\subset G_i\right\}
$ \textit{vertical thickness} of $E$.
In other words, the vertical thickness of an edge $E$ is
a difference of the maximum element and the minimum element
of the set
$$
\left\{w_2\in\{0,\dots,N\}\,\vert\,\Val(X,Y;f_\ep)=\val(a_w)+w_1X+w_2Y,\ \ \ 
\forall (X,Y)\in E
\right\}.
$$
\end{defi}
For example, the vertical thickness of a vertical edge is $0$.

\begin{lemma}
The vertical thickness of $E$ equals to the degree of the projection
\[
V(f^P)\to\CC\setminus\{0\};\ (x,y)\mapsto x,\qquad P\in \mathrm{Int}\,E.
\]
\end{lemma}
\begin{proof}
First note that the truncation
$f^P$ $(P\in\mathrm{Int}\,E)$ is determined uniquely
by Lemma \ref{ex-lemma-0.4}.
The Laurent polynomial $f^P$ is of the form $f^P=\sum c_wx^{w_1}y^{w_2}$,
where $w=(w_1,w_2)$ runs over the set $\{w\,\vert\,\Val(X,Y;f_\ep)=\val(a_w)
+{w_1}X+{w_2}Y,\ 
\forall (X,Y)\in E\}$ and $c_w$ is a non-zero complex number for any $w$.
Because the degree of the projection $x:V(f^P)\to \CC\setminus\{0\}$ equals to
the difference between the maximum degree
and the minimum degree w.r.t.\,$y$ consisted in $f^P$,
the desired result is obtained.
\end{proof}
\bigskip

Let $C_\ep^T:=V(f_\ep(y,x))$
be the curve
obtained from $C_\ep$ by switching the $x$ and $y$ coordinates.
For an edge $E\subset \mathrm{Trop}\,C$, we define
the \textit{horizontal thickness} 
of $E$
by the vertical thickness of 
$
E^T\subset \mathrm{Trop}\,C^T$ which is the image of $E$
by the morphism $(X,Y)\mapsto (Y,X)$.
For example, the horizontal thickness of horizontal edges is $0$.

\subsubsection{Multiplicity}
As in example (I\!I) in
Section \ref{sec3.1} above, it may happen that
more than one cylinders (or horns) are
associated with one edge (or one leaf).
\begin{defi}
Let $E\subset\mathrm{Trop}\,C$ be an edge $($resp.\,\,a leaf $)$.
The multiplicity of $E$ is the number of cylinders $($resp.\,\,horns $)$
associated with $E$.
\end{defi}

Let $E\subset\mathrm{Trop}\,C$ be an edge (resp.\,\,a leaf) of multiplicity $m$,
and $P$ be a point in $\mathrm{Int}\,E$.
By the approximation theorem \ref{approximationtheorem},
the variety $V(f^P)$ must be decomposed into $m$ irreducible components:
$V(f^P)=V(f^P_1)\coprod\cdots\coprod V(f^P_m)$, where
$f^P=f^P_1\cdots f^P_m$.
We often regard the edge $E$
as the union of distinguished
$m$ edges:
$E=E_1\amalg\dots\amalg E_m$ (See Figure \ref{fig7}
in Section \ref{sec4}),
where each
$E_i$ corresponds to $V(f^P_i)$ ($P\in\mathrm{Int}\,E$) and is taken to be
of multiplicity one.

Define the vertical thickness of $E_i$ as the degree of the projection:
$
V(f^P_i)\,\to\, \CC\setminus\{0\}\ ;\
(x,y)\mapsto x$.
Denote the vertical thickness of $E_i$ by $q_i$.
Naturally, $q_1+\dots+q_m$ equals to the vertical thickness of $E$.
Similarly, we can
define the horizontal thickness of $E_i$ as the vertical thickness of $E_i^T$.

The following definition is given for the next section.
\begin{defi}
Let $L_{i,j}$ be the vertical edge which connects 
$(B_{i,j},\mathcal{N}_i(B_{i,j}))$ and $(B_{i,j},\mathcal{N}_{i+1}(B_{i,j}))$.
The \textit{ceiling of} $L_{i,j}$ 
is the set $G_{i+1}=\{(X,Y)\,\vert\,Y=\mathcal{N}_{i+1}(X)\}$
and the \textit{floor of} $L_{i,j}$ is the set 
$G_i=\{(X,Y)\,\vert\,Y=\mathcal{N}_{i}(X)\}$.
\end{defi}

\subsection{Regularity of tropical curves}

Let $P=(X_0,Y_0)$ be a point in $\mathrm{Trop}\,C=TV(f_\ep)$, where
$f_\ep\!=\!\sum_{w=(w_1\!,\!w_2)\in\ZZ^2}{a_wx^{w_1}y^{w_2}}$.
Considering the set 
$\Theta(P)=\{w\in\ZZ^2\,\vert\,\Val(X_0,Y_0;f_\ep)=\val(a_w)+w_1X_0+w_2Y_0\}$
(Section \ref{sec2.2}),
we have $\sharp\Theta(P)\geq 2$ (Remark \ref{rem-ex}).
More precisely, the number of elements of $\Theta(P)$ satisfies the following 
inequalities: 
\begin{equation}\label{regularity}
\begin{array}{cl}
\text{i) }&
P \text{ is an inner point of some edge of } \mathrm{Trop}\,C\ \Rightarrow\
\sharp\Theta(P)\geq 2,\\ 
\text{ii) }& P \text{ is a vertex of } \mathrm{Trop}\,C\ \Rightarrow\
\sharp\Theta(P)\geq 3. 
\end{array}
\end{equation}
These inequalities reflect the fact that
the intersection of (generic)
two planes is an line and the intersection of (generic) three planes is a point
in $\RR^3$.

In the present paper, we often assume some genericness condition on the defining
polynomial of $\mathrm{Trop}\,C$.
\begin{defi}
The tropical plane curve $TV(f_\ep)$ is regular if 
the equalities in $(\ref{regularity})$ hold.
\end{defi}

\section{Integration Theory}\label{sec4}
The integration theory over tropical curves was first introduced in \cite{Mik}.
Hereafter we will show
that the ultradiscrete limit of holomorphic integrals over
$C_\ep$ coincides with the holomorphic integral over $\mathrm{Trop}\,C$
(for $C_\ep$ of some type).

\subsection{Definition of the holomorphic integral over tropical curves}

In this section, we give a brief introduction to integration theory
over tropical curves, following \cite{Takenawa,Mik}.
\bigskip

We first equip $\mathrm{Trop}\,C$ with the structure of a \textit{metric graph}.
Let $E$ be an edge of $\mathrm{Trop}\,C$.
$E$ has the expression $E=\{(X_0,Y_0)+t(u,v)\,\vert\,0\leq t\leq \ell\}$,
$(u,v\!\in \ZZ)$.
It can be assumed that $u$ and $v$ are coprime without loss of generality.
We call the vector $(u,v)$ the \textit{primitive vector} of $E$.
We define a tropical length $\ell_T$ of $E$ by 
$\ell_T(E):=\ell$.
With this length the tropical curve $\mathrm{Trop}\,C$ becomes a metric graph.

The metric on $\mathrm{Trop}\,C$ defines a symmetric bilinear form $\ell_T(\cdot,\cdot)$
on the space of paths in $\mathrm{Trop}\,C$.
For this,
we define $\ell_T(\Gamma,\Gamma):=\ell_T(\Gamma)$ for non-self-intersecting
path $\Gamma$,
and extend it to any pairs of paths bilinearly.
Figure \ref{fig4} shows an example of $\ell_T(\cdot,\cdot)$.
Note that the number
$\vert\ell_T(\Gamma_1,\Gamma_2)\vert$ equals the tropical length 
$\ell_T(\Gamma_1\cap\Gamma_2)$. 
This bilinear form gives 
the tropical length of intersection of two paths up to sign.
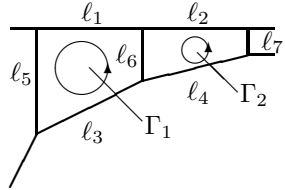
\begin{figure}[htbp]
\begin{center}
\begin{picture}(100,70)
\thicklines
\put(0,60){\line(1,0){100}}
\put(10,20){\line(0,1){40}}
\put(10,20){\line(-1,-2){10}}
\put(10,20){\line(2,1){40}}
\put(50,40){\line(4,1){40}}
\put(90,50){\line(1,0){10}}
\put(90,50){\line(0,1){10}}
\put(50,40){\line(0,1){20}}
\put(27,63){$\ell_1$}
\put(67,63){$\ell_2$}
\put(27,18){$\ell_3$}
\put(67,33){$\ell_4$}
\put(0,40){$\ell_5$}
\put(40,47){$\ell_6$}
\put(95,52){$\ell_7$}
\thinlines
\put(27,45){\circle{20}}
\put(37,47){\vector(0,1){0}}
\put(70,52){\circle{10}}
\put(75,55){\vector(0,1){0}}
\put(30,45){\line(1,-1){20}}
\put(71,51){\line(1,-1){15}}
\put(51,20){$\Gamma_1$}
\put(87,32){$\Gamma_2$}
\end{picture}
\end{center}
\caption{Example of a metric graph.
We have $\ell_T(\Gamma_1,\Gamma_1)=\ell_1+\ell_3+\ell_5+\ell_6$,
$\ell_T(\Gamma_2,\Gamma_2)=\ell_2+\ell_4+\ell_6+\ell_7$ and 
$\ell_T(\Gamma_1,\Gamma_2)=\ell_T(\Gamma_2,\Gamma_1)=-\ell_6$.
}
\label{fig4}
\end{figure}

Let $g$ be the genus of $\mathrm{Trop}\,C$ and choose
a homology basis 
$$
T_{\beta_1},\dots,T_{\beta_g}\in H_1(\mathrm{Trop}\,C\,;\,\ZZ).
$$
A \textit{tropical period matrix} $B_T$ is
the $g\times g$ matrix defined by $B_T:=(\ell_T(T_{\beta_i},T_{\beta_j}))_{i,j}$.
Since $\ell_T$ is non-degenerate, $B_T$ is symmetric and positive definite.
\bigskip

Here, we note the relation between the tropical length, the multiplicity and
the vertical thickness 
of the edge in $\mathrm{Trop}\,C$.
Let 
\[E:=\{((1-t)X_0+tX_1,(1-t)Y_0+tY_1)\,\vert\,0\leq t\leq 1,
X_0\lneq X_1\}\]
be a non-vertical edge of
vertical thickness $q$ and of horizontal thickness $w$. 
From the definition of vertical thickness and horizontal thickness,
it follows that
$E$ is part of the line
defined by the equation:
$aX+bY+c=(a+w)X+(b\pm q)Y+c'$
(see Definition \ref{def3.3.1}).
\begin{lemma}\label{lemma4.1}
Let $\xi:=\mathrm{g.c.d.}(q,w)$.
Then the tropical length of $E$ is expressed as:
\[
\ell_T(E)=\frac{\xi}{q}(X_1-X_0).
\]
\end{lemma}
\begin{proof}
We can assume $(X_0,Y_0)=(0,0)$ by translation.
Then we obtain $wX_1=\pm qY_1$.
Let $\eta:=\mathrm{g.c.d.}(X_1,Y_1)$ and
\[
\left\{
\begin{array}{l}
X_1=x\eta\\
Y_1=y\eta
\end{array}
\right.,\qquad
\left\{
\begin{array}{l}
q=\mu\xi\\
w=\nu\xi
\end{array}
\right..
\]
Then, we conclude $\mu=x$ and $\nu=y$ by elementary arguments.
From the
definition of the tropical length we obtain
\begin{equation*}
\ell_T(E)=\mathrm{g.c.d.}(X_1,Y_1)=\eta=\frac{X_1}{x}=\frac{\xi}{q}X_1.
\end{equation*}
\end{proof}
\bigskip

Next, we introduce \textit{affine transformations} of tropical curves.
Let $\mathcal{T}$ be a tropical curve in $\RR^2$.
For a $2\times 2$ matrix 
$\theta=
\left(\begin{array}{@{\,}cc@{\,}}
\alpha & \beta \\
\gamma & \delta
\end{array}\right)
$, such that
$\alpha\delta-\beta\gamma=1$,
we have a new set
\[
\mathcal{U}:=\{\theta (X,Y)^T\in\RR^2\,\vert\,(X,Y)\in\mathcal{T}\}.
\]
For the complex curve $\{f_\ep(x,y)=0\}$,
this transformation associates with the translation $x\mapsto x^\delta y^{-\beta}$
$y\mapsto x^{-\gamma}y^\alpha$, which is invertible and holomorphic.
In particular, $\mathcal{U}$ is also a tropical curve associated with 
$f_\ep(x^\delta y^{-\beta},x^{-\gamma}y^\alpha)=0$.

Concerning such affine translations, we have the following fundamental result.

\begin{prop}
Let $\theta\in M_2(\ZZ)$ be a $2\times 2$ matrix with $\det{\theta}=1$. Then
\begin{enumerate}
\def\labelenumi{$($\theenumi$)$}
\def\theenumi{\roman{enumi}}
\item The length of an edge is $\theta$-invariant.
\item For an edge $L\in\RR^2$, there exists $\theta\in M_2(\ZZ)$
such that $\theta\cdot L$ is vertical.
\end{enumerate}
\end{prop}
\begin{proof}
(i) Let $(u,v)$ be a primitive vector of the edge $L\in\RR^2$.
It is sufficient to prove that the image $\theta (u,v)^T=(\alpha u+\beta v,
\gamma u+\delta v)$ is also primitive.
For this, we have only to prove $\mbox{g.c.d.}(\alpha u+\beta v,
\gamma u+\delta v)=1$.
This can be proved by elementary methods and we omit the proof.\\
(ii) For the primitive vector $(u,v)$ of $L$,
it is enough to define 
$
\theta=
\left(\begin{array}{@{\,}cc@{\,}}
v & -u \\
w & z
\end{array}\right)
$
such that $vz+wu=1$. 
\end{proof}

\begin{rem}
The vertical 
and the horizontal
thickness of an edge depend on the coordinate functions $X$ and $Y$.
\end{rem}

\subsection{Main theorem}\label{sec4.3}

Now we proceed to integration theory over 
$C_\ep=V(f_\ep)$.
In order to make the problem easier,
we deal only with the case where: 
\begin{enumerate}
\def\labelenumi{\theenumi)}
\def\theenumi{\roman{enumi}}
\item $f_\ep$ is non-singular and totally non-singular (Section \ref{sec2}),
\item $C_\ep$ is non-degenerate,
\item $\mathrm{Trop}\,C$ is regular.
\end{enumerate}
Conditions i)--iii) and the \textit{genericness condition} (Section \ref{sec3})
lead to the following properties
(see Appendix):
\begin{enumerate}
\def\labelenumi{\theenumi)}
\def\theenumi{\roman{enumi}}
\setcounter{enumi}{3}
\item for each edge $E$, 
$m=\mathrm{g.c.d.}(q,w)$,
where $m,q,w$ are respectively
the multiplicity, the vertical thickness and the horizontal thickness of $E$,
\item for each edge $E=E_1\amalg\dots\amalg E_m$,
$q_1=\dots=q_m$, $w_1=\dots=w_m$, where $q_i,w_i$ are the vertical thickness 
and the horizontal thickness
of $E$.
\end{enumerate}

\begin{rem}
The curves $C$ introduced in
Examples $(\mathrm{I})$ and $(\mathrm{II})$ in Section $\ref{sec3.1}$
satisfy these conditions.
\end{rem}
We say that $C_\ep$ has a \textit{good tropicalization} if 
$C_\ep$ and $\mathrm{Trop}\,C$ satisfy the genericness condition in Section \ref{sec3}
and
conditions i)--iii) above.
\begin{rem}
The conditions i) and ii) are necessary conditions to construct
the integration theory. 
The condition iii) is required for simplicity
of the calculations.
(The author is not sure whether the condition iii) can be omitted.)
\end{rem}
\bigskip

By the approximation theorem \ref{approximationtheorem},
there exists small $\zeta>0$ such that all $C_\ep$ $(\ep\in (0,\zeta))$
are homotopic. 
Hereafter $\ep$ denotes a small real number which satisfies $0<\ep<\zeta$
unless otherwise is stated.

Let $g$ be the genus of $C_\ep$. 
Define homology cycles
$\alpha_1,\alpha_2,\dots,\alpha_g\in H_1(C_\ep;\ZZ)$
as in Figure \ref{fig6}.
Any cycle is associated with a long cylinder connects
two sub-surfaces.
Next define the
homology cycles $\beta_1,\beta_2,\dots,\beta_g\in H_1(C_\ep;\ZZ)$
such that the intersection index $\alpha_i\circ\beta_j$
is $\delta_{i,j}$,
in a canonical way.
We assume the cycles $\alpha_i=\alpha_i(\ep)$ and $\beta_i=\beta_i(\ep)$
are continuous with respect to $\ep$.
Denote the normalised holomorphic differentials over $C_\ep$
by $\omega_1,\omega_2,\dots,\omega_g\,$ $(\int_{\alpha_i}{\omega_j}=\delta_{i,j})$.
The \textit{period matrix} $B_\ep$ of the Riemannian surface $C_\ep$ is the
$g\times g$ matrix defined by
$B_\ep:=(\int_{\beta_i}{\omega_j})_{i,j}$. 
\begin{figure}[htbp]
\begin{center}
\includegraphics[bb=18 487 585 811, clip,
width=12cm
]{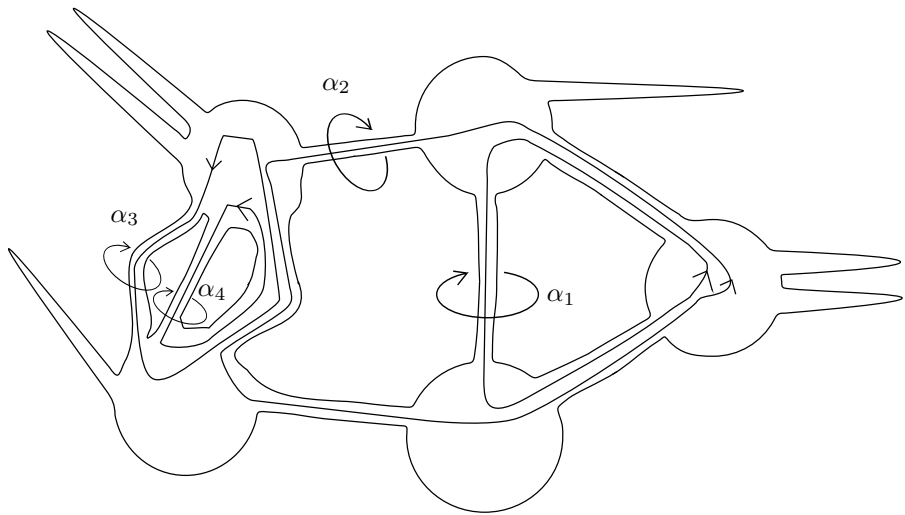}
\begin{picture}(0,180)(225,10)
\put(0,170){$\alpha_2$}
\put(85,90){$\alpha_1$}
\put(-80,120){$\alpha_3$}
\put(-47,92){$\alpha_4$}
\end{picture}
\end{center}
\caption{An example of the definition of $\alpha_1,\dots,\alpha_g;\beta_1,\dots,\beta_g$
in $H^1(C\,;\,\ZZ)$.
We always consider that an $\alpha$-cycle surrounds a cylinder.}
\label{fig6}
\end{figure}
\bigskip

Now we state the main theorem of this paper. 
Take a homology basis $T_{\beta_1},\dots,$
$T_{\beta_g}\in H_1(\mathrm{Trop}\,C\,;\,\ZZ)$
associated with a homology basis $\beta_1,\dots,\beta_g\in H_1(C\,;\,\ZZ)$.
If more than one cylinder can be associated with one tropical edge,
we say that hidden edges and
cycles exist here (see Figure \ref{fig7}).
When we define homology cycles in $H_1(\mathrm{Trop}\,C;\ZZ)$,
these edges must be distinguished.
\begin{figure}[htbp]
\begin{center}
\includegraphics[bb=84 520 552 806, clip,
width=7cm]
{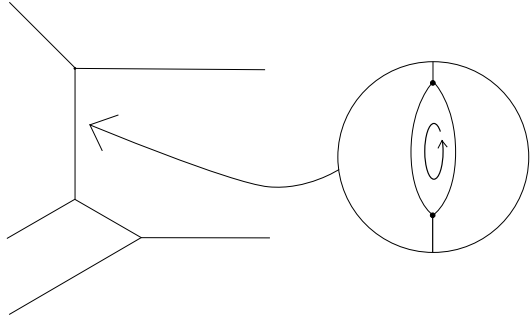}
\end{center}
\caption{If there exist an edge with multiplicity $m>1$,
we regard this edge as the union of $m$ edges
of multiplicity one.
}
\label{fig7}
\end{figure} 
Recall that
we regard
an edge $E$ of multiplicity $m$ as the union:
$E=E_1\amalg E_2\amalg\dots\amalg E_m$.
Let $B_T:=(\ell_T(T_{\beta_i},T_{\beta_j}))_{i,j}$ be a \textit{period matrix of} 
$\mathrm{Trop}\,C$.

\begin{theorem}\label{thm}
If $C_\ep$ has a good tropicalization, then
\[
B_\ep\sim\frac{-1}{2\pi \ai\ep}B_T\quad\quad (\ep\to 0).
\]
\end{theorem}

The rest of this paper is devoted to
the proof of this theorem.

\subsubsection*{Preliminaries for integral calculus}

In this section, we study integral calculus over $V(f_\ep)$
and the asymptotic behaviour of integrals when $\ep$ tends to zero.
For a pL-polynomial $f_\ep$, let us define the variety
$\tilde{V}(f):=\{(x,y,\ep)\in\CC\RR^2\times(0,1)\,\vert\,f_\ep(x,y)=0\}$.
Denote the natural embedding $\CC\RR^2\hookrightarrow\CC\RR^2\times(0,1)$\
;\ $(x,y)\mapsto(x,y,\ep)$
by $j_\ep$.
Naturally it follows that $j_\ep^{-1}(\tilde{V}(f))=V(f_\ep)$.

Let $\mathcal{U}\subset \tilde{V}(f)$ be a simply connected domain
and $\omega_\ep$ be a $1$-form over $\mathcal{U}$ such that
i) $\omega_\ep$ is a holomorphic differential over $j_\ep^{-1}(\mathcal{U})=
\mathcal{U}\cap V(f_\ep)$,
and ii) $\omega_\ep$ is continuous with respect with $\ep$.
By elementary arguments in complex analysis, we can prove the existence
of a primitive function $\Omega_\ep$ of $\omega_\ep$.
The integration of $\omega_\ep$ along a smooth path
$[0,1]\to j_\ep^{-1}
(\mathcal{U})\,;\,\theta\mapsto \gamma_\ep(\theta)$ is defined by the
formula $\int_{\gamma_\ep}{\omega_\ep}:=\Omega_\ep(\gamma_\ep(1))-
\Omega_\ep(\gamma_\ep(0))$.

Let $(0,1)\times[0,1]\mapsto \tilde{V}(f)\,;\,(\ep,\theta)\mapsto 
\gamma_\ep(\theta)$ be a smooth map such that $\gamma_\ep(\theta)\in V(f_\ep)$
for all $\ep$ and $\theta$.
Our aim in 
this section is to evaluate the asymptotic behaviour of the value
$\int_{\gamma_\ep}{\omega_\ep}$ when $\ep$ goes to zero.
\bigskip

Due to the surjectivity theorem \ref{surjectivitytheorem},
more detailed information can be added to the definition of a path
on $\tilde{V}(f)$.
Let $v_1=(X_1,Y_1)$ and $v_2=(X_2,Y_2)$
be two points in $\mathrm{Trop}\,C=TV(f_\ep)$ and 
$\mathcal{R}^{v_i}f_\ep=f^{v_i}+\Delta^{v_i}$ $(i=1,2)$
be the canonical expressions of $f_\ep$ at $v_i$.

Consider a path $\gamma':[0,1]\to V(f^{v_1})$ on the variety $V(f^{v_1})$.
By the surjectivity theorem, there exists the smooth map 
$(0,1)\times[0,1]\to \tilde{V}(f)\,;\, (\ep,\theta)\mapsto \gamma_\ep(\theta)$
such that 
\[
\gamma_\ep(\theta)=(x_\ep(\theta)\,\ee^{X_1},y_\ep(\theta)\,\ee^{Y_1})\in V(f_\ep),
\quad\text{ and }\quad
\lim_{\ep\to 0^+}(x_\ep(\theta),y_\ep(\theta))\to \gamma'(\theta),
\ (\forall \theta)
,\]
where $x_\ep(\theta),y_{\ep}(\theta)\in R^\times$ for any $\theta$.
We often abbreviate the above
notation as $\gamma(\theta)=(x\ee^{X_1},y\ee^{Y_1})$ $x,y\in R^\times$
if there is no chance of confusion.

Similarly, 
if $\tilde{V}(f)$ is connected,
for two points $(x_1,y_1)\in V(f^{v_1})$ and $(x_2,y_2)\in V(f^{v_2})$,
there exists a smooth map
$(0,1)\times[0,1]\to \tilde{V}(f)\,;\, (\ep,\theta)\mapsto \gamma_\ep(\theta)$
such that 
\[
\gamma_\ep(0)=(x_{1,\ep}\ee^{X_1},y_{1,\ep}\ee^{Y_1}),\quad
\gamma_\ep(1)=(x_{2,\ep}\ee^{X_2},y_{2,\ep}\ee^{Y_2}),\]
and
$\lim_{\ep\to 0^+}(x_{i,\ep},y_{i,\ep})\to (x_i,y_i)$.
We often use the notation 
$\int_{\gamma(0)}^{\gamma(1)}{\omega}$ instead of $\int_{\gamma_\ep}{\omega_\ep}$
if the meaning is clear.

\subsubsection{Approximation of integral calculus}

In the rest of this paper, $b_{i,j}$ denotes $\int_{\beta_j}{\omega_i}$.

Let $\mathfrak{S}$ be the set of $1$-forms over $\tilde{V}(f_\ep)$
such that i) $\omega_\ep\in\mathfrak{S}$
is a meromorphic differential over $j_\ep^{-1}(\tilde{V}(f))=V(f_\ep)$,
ii) $\omega_\ep\in\mathfrak{S}$ is continuous with respect to $\ep$.
Define the subsets $\mathcal{M}$ and $\mathcal{F}$ by the formulae
\begin{gather*}
\textstyle
\mathcal{M}:=\!\{\omega_\ep\in\mathfrak{S}\,\vert\,\int_{\beta_k}{\omega_\ep}
=\sum_{i,j=1}^g{c_{i,j}^k(\ep)\,b_{i,j}},\text{ where } 
-\lim_{\ep\to 0^+}{\ep\log c_{i,j}^k}>0,\ \ \forall i,j,k
\},\\
\textstyle
\mathcal{F}:=\{\omega_\ep\in\mathfrak{S}\,\vert\,
\lim_{\ep\to 0^+}{\vert\int_{\beta_i}{\omega_\ep}\vert}<+\infty,\ \ \forall i\}.
\end{gather*}
It is clear that they are $R^\times$-vector spaces.
\begin{rem}
If the main theorem \ref{thm} is true, is follows that
$\mathcal{M}\subset\mathcal{F}$.
\end{rem}
\bigskip

We say a differential over Riemannian surface 
is \textit{of the first kind} if it has no singularity,
\textit{of the second kind} if it has poles without residue and
\textit{of the third kind} if it has poles with non-zero residue.
For the proof of the main theorem, we start with differentials of the third kind.
Let $P_+,P_-$ be two points on $C_\ep$.
A smooth curve $C_\ep$ has the normalised
differential of the third kind
$\omega_{P_+-P_-}=\omega_{P_+-P_-}(\ep)$,
possessing simple poles with
residue $+1/(2\pi \ai)$ at $P_+$ and $-1/(2\pi \ai)$ at $P_-$,
and holomorphic
over $C_\ep\setminus \{P_+,P_-\}$ satisfying
$\int_{\alpha_i}{\omega_{P_+-P_-}}=0\ (i=1,2,\dots,g)$.
Generally, for $n$ points
$P_1,\dots,P_n\in C_\ep$ and complex numbers $c_1+\dots+c_n=0$,
there is a unique normalised
differential $\omega_{c_1P_1+\dots+c_nP_n}$,
with residue $c_i/(2\pi \ai)$ at $P_i$ $(i=1,\dots,n)$.

Recall that a point in $\{x=\infty,0\}\cup\{y=\infty,0\}$
is associated with a leaf in $\mathrm{Trop}\,C$. (cf. Section
\ref{sec3.1}).
\begin{lemma}\label{lemma4.2}
Let $P_+,P_-\in C_\ep$ are two points
in $\{x=\infty\}\cup\{x=0\}\cup\{y=\infty\}
\cup\{y=0\}
$
which are associated with the \textit{same leaf} in $\mathrm{Trop}\,C$.
Then it follows that 
$\omega_{P_+-P_-}\in\mathcal{M}+\mathcal{F}$.
\end{lemma}
\begin{proof}
Let $L\subset\mathrm{Trop}\,C$ be the leaf which includes $P_{\pm}$.
By rotation, 
$L$ can be assumed to be vertical tending to
$Y= +\infty$: 
$
L=\{(B,Y)\,\vert\,Y\geq \mathcal{N}_N(B)\}
$,
where $\mathcal{N}_N(X)$ is a tropical function as
defined in Section \ref{sec3}.
Hereafter, we denote this function by $\mathcal{N}(X)$.

Let 
$\textstyle
f_\ep(x,y)=\sum_{i=0}^N{a_i(x)y^{N-i}},
a_N(x)=c\ee^{A}x^m\prod_{j=1}^d
(x-u_j\ee^{B_j}),
$ 
$(c,u_j\in \rbatu\!,m\in\NN_{\geq 0},
A,B_j\in \QQ_{\geq 0})$
be the defining polynomial of $C_\ep$.
By the above assumption 
the $y$-coordinate of $P_\pm$ equals $0$.
Then
the $x$-coordinate of $P_\pm$ can be written as
$
x(P_+)=u_{k_1}\ee^{B}
$,
$x(P_-)=u_{k_2}\ee^B
$
for some
$k_1,k_2\in\{1,2,\dots,d\}$
such that
$B=B_{k_1}\!\!=B_{k_2}$.

Now we proceed to integration calculus.
Consider the polynomial 
$\phi(x)$ defined by
\begin{equation}\label{eq4.5}
\frac{1}{2\pi \ai}\left\{\frac{1}{x-u_{k_2}\ee^{B}}-
\frac{1}{x-u_{k_1}\ee^{B}}
\right\}
=\frac{\phi(x)}{a_N(x)}.
\end{equation}
Define the new differential $\omega_\mathrm{f}$ by
\begin{equation}\label{eq4.3}
\omega_{\mathrm{f}}:=\frac{\phi(x)\, dx}{y\,f_y(x,y)},\quad\quad\quad
\text{where } 
f_y(x,y):=\partial_y f_{\ep}(x,y)
=\textstyle\sum_{i=0}^{N-1}{(N-i)\,a_{i}(x)\,y^{N-i-1}}.
\end{equation}
The singularity of $\omega_\mathrm{f}$
must be contained in $\{x=\infty\}\cup\{y=\infty,0\}$.
($\because$
$dx/f_y$ is always holomorphic over smooth plane curves).

The following sublemmas describes the behaviour of the differential
$\omega_{\mathrm{f}}$.

\begin{sub}
The distribution of residues of $\omega_{\mathrm{f}}$ is given as follows: \\
i) $\mathrm{Res}_{P_+}\omega_{\mathrm{f}}=+1/(2\pi \ai)+o(\ee^0)$,\qquad 
ii) $\mathrm{Res}_{P_-}\omega_{\mathrm{f}}=-1/(2\pi \ai)+o(\ee^0)$,\qquad\\
iii) $\mathrm{Res}_P\,\omega_{\mathrm{f}}=o(\ee^0),$\quad $P\neq P_\pm$.
\end{sub}

\begin{proof}
i)--ii) Let $v\in\mathrm{Trop}\,C$ be the vertex which is at the foot of $L$
and let
$\Omega$ be the sub-surface associated with $v$.
We take small cycles
$\gamma_\pm\subset\Omega$
which loop around $P_\pm$ anti-clockwise and which satisfy
\begin{equation}\label{eq4.4}
(x,y)\in\gamma_\pm\ \Rightarrow\ x=r\ee^{B}, y=s\ee^{\mathcal{N}(B)},\quad
\exists r,s\in \rbatu.
\end{equation}
Denote the vertical thickness of the floor of $L$ by $q'$.
On $\gamma_\pm\subset\Omega$, 
the dominant terms of $f_\ep(x,y)=\sum{a_i(x)\,y^{N-i}}$ 
are $a_N$ and $a_{N-q'}y^{q'}$.
Then the dominant term of $f_y(x,y)=\sum{(N-i)a_i(x)y^{N-i-1}}$
is $q'a_{N-q'}y^{q'-1}$.
Hence, on $\gamma_\pm$,
one has
\begin{equation}\label{eq4.9c}
y\,f_y\,\sim\, q'a_{N-q'}\,y^{q'}
=-q'a_N+\cdots.
\end{equation}
For the second equation in (\ref{eq4.9c}), 
we used $f_\ep(x,y)=0$.
Then, we obtain
\begin{equation}\label{eq4.6}
\int_{\gamma_\pm}{\omega_{\mathrm{f}}}
\sim\int_{\gamma_\pm}{\frac{\phi(x)\,dx}{-q'a_N(x)}}.
\end{equation}
We claim that the integral on the right hand side
takes the value $\mp 1$. 
To prove this, we recall the relation
\[
(x,y)\in\Omega
\quad\Rightarrow\quad
0=f_\ep(x,y)=a_{N-q'}y^{q'}+a_N+o(\ee^{q'\mathcal{N}(B)+\val(f_{N-q'})}).
\]
Hence $q'$ of 
the zeros 
of $a_N(x)$
satisfy the equation
$(x,y)=(u_{k_1}\ee^B,o(\ee^{\mathcal{N}(B)}))$,
which implies these points $(x,y)$ are
in the horn containing $P_+$.
(There
also exist $q'$ points satisfying $x=u_{k_2}\ee^{B},y=o(\ee^{\mathcal{N}(B)})$
in the horn containing $P_-$).
The circles $\gamma_+$ and $\gamma_-$ encircle these $q'$ points 
respectively. (Figure \ref{fig9}).
By definition of $\phi(x)$ (\ref{eq4.5}),
the residues of $-\phi(x)/(q'a_N(x))$ 
equal $\pm 1/q'$ at each pole.
Therefore, by the residue theorem, we obtain
\begin{equation}
\textstyle \int_{\gamma_+}{\omega_{\mathrm{f}}}=1+o(\ee^0),\quad\quad
\int_{\gamma_-}{\omega_{\mathrm{f}}}=-1+o(\ee^0).
\end{equation}
\begin{figure}[htbp]
\begin{center}
\includegraphics[bb=100 500 302 826, clip,
width=4cm]
{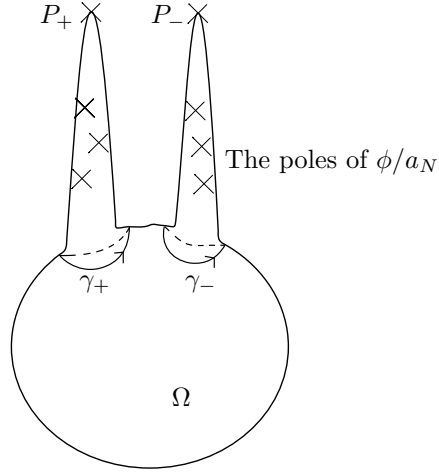}
\begin{picture}(0,0)
\put(-100,175){$P_+$}
\put(-58,175){$P_-$}
\put(-30,120){The poles of $\phi/a_N$}
\put(-85,75){$\gamma_+$}
\put(-45,75){$\gamma_-$}
\put(-50,30){$\Omega$}
\end{picture}
\end{center}
\caption{A sketch of $\Omega$. The sphere $\Omega$ has
finitely many horns. Each horn is associated with $E_i$.}
\label{fig9}
\end{figure}

(iii)
Let $P\in(\{x=\infty,0\}\cup\{y=\infty,0\})\setminus\{P_+,P_-\}$.
Defining a circle $\gamma$ in the appropriate sub-surface,
we can assume that $\gamma$ surrounds one horn containing $P$.
Let $L'$ be a leaf and
$q''$ be the vertical thickness of $L'$.
The leaf $L'$
:(a) contained in $G_N=\{Y=\mathcal{N}(X)\}$ and $\gamma$ does not surround any pole
of $\omega_{\mathrm{f}}$
or (b) contained in $\{Y<\mathcal{N}(X)\}$.
In case (b), the dominant term of $f_\ep$ is neither $a_N$ nor $a_{N-q'}y^{q'}$,
and so there exists positive $\delta$ such that $\zet{a_N}<\ee^{\delta}\zet{f_y}$.
Therefore we have
\begin{equation}\label{eq4.11a}
\textstyle
\int_{\gamma}{\omega_{\mathrm{f}}}\sim
\left\{
\begin{array}{ll}
\int{\{\phi(x)/(-q''a_N(x))\}\,dx} & 
(L'\subset G_N)\\
o(\ee^0) & (\mbox{otherwise})
\end{array}\right.,
\end{equation}
which yields $\int_{\gamma}{\omega_{\mathrm{f}}}=o(\ee^0)$
(cf. (\ref{eq4.3})). 
\end{proof}

\begin{sub}
i) $\int_{\alpha_i}{\omega_\mathrm{f}}=o(\ee^0)\quad\forall i$,\qquad
ii) $\omega_{\mathrm{f}}\in\mathcal{F}$.
\end{sub}
\begin{proof}
i)
Let $E$ be the edge associated with $\alpha_i$.
The edge $E$
: (a) contained in $G_N=\{Y=\mathcal{N}(X)\}$ and $\alpha_i$ does not surround any pole
of $\omega_{\mathrm{f}}$
or (b) contained in $\{Y<\mathcal{N}(X)\}$.
In each case, the Residue theorem and (\ref{eq4.11a})
immediately leads to
the evaluation:
$\int_{\alpha_i}{\omega_\mathrm{f}}=o(\ee^0)$.

ii) 
Let us decompose $\beta_i=\beta_i(\ep)$ into finitely many parts.
Denote one of these portions by $\gamma_\ep(\theta)$.
It is sufficient to prove the finiteness of the limit of integral
for arbitrary simply connected region $\mathcal{U}\in \tilde{V}(f)$
and arbitrary path $\gamma_\ep(\theta):(0,1)\times[0,1]\to\mathcal{U}$.
Moreover, we can assume $j_\ep^{-1}(\mathcal{U})$ is contained in some cylinder.
Let $E$ be the associated edge.
If $E\not\subset G_N$, it follows that $\int_{\gamma_\ep}{\omega_{\mathrm{f}}}=
o(\ee^0)$ by (\ref{eq4.11a}).
Let $E\subset G_N$.
Denote the vertical thickness of $E$ by $q$.
Let $x=r_1\ee^{X_1}$ be the $x$-coordinate of $\gamma_\ep(0)$
and $x=r_2\ee^{X_2}$ the $x$-coordinate of $\gamma_\ep(1)$. 
Then,
\begin{align*}
\textstyle
\int_{\gamma_\ep}{\omega}=
&\textstyle
\int_{r_1\ee^{X_1}}^{r_2\ee^{X_2}}{\{\phi(x)/(-q\,a_N(x))\}\,dx}+o(\ee^0)\\
&=\textstyle
(-2q\pi \ai)^{-1}\int_{r_1\ee^{X_1}}^{r_2\ee^{X_2}}
{\{(x-u_{k_1}\ee^B)^{-1}-(x-u_{k_2}\ee^B)^{-1}\}\,dx}
+o(\ee^0)\\
&=(-2q\pi \ai)^{-1}\log{\left\{\frac{r_2\ee^{X_2}-u_{k_1}\ee^B}{r_1\ee^{X_1}-u_{k_1}\ee^B}
\cdot\frac{r_1\ee^{X_1}-u_{k_2}\ee^B}{r_2\ee^{X_2}-u_{k_2}\ee^B}\right\}}+o(\ee^0)\\
&=(-2q\pi \ai)^{-1}\log{\left\{\frac{r_1\ee^{X_1}-u_{k_2}\ee^B}{r_1\ee^{X_1}-u_{k_1}\ee^B}
\cdot\frac{r_2\ee^{X_2}-u_{k_1}\ee^B}{r_2\ee^{X_2}-u_{k_2}\ee^B}\right\}}+o(\ee^0).
\end{align*}
Recalling that $\ee=e^{-1/\ep}$, we conclude that
the expression in the last line converges to the finite number 
\begin{center}
$(-2q\pi \ai)^{-1}$, $(-2q\pi \ai)^{-1}\log{(u_{k_1}/u_{k_2})}$, \
or \ $(-2q\pi \ai)^{-1}\log{(u_{k_2}/u_{k_1})}$
\end{center}
when $\ep\to 0^+$.
\end{proof}

In the
final step of the proof of Lemma \ref{lemma4.2},
we use the Riemann bilinear relation \cite{Fay}:
\begin{gather}
\textstyle
\sum_{i=1}^g{(A'_iB_i-A_iB_i')}=2\pi \ai\ \cdot
\sum_j{\mathrm{Res}_{P_j}(\omega^{(3)})
\cdot\int^{P_j}{\omega^{(1)}}}\label{eq4.11}\\
\textstyle
A_i=\int_{\alpha_{i}}{\omega^{(3)}},\quad
B_i=\int_{\beta_{i}}{\omega^{(3)}},\quad
A'_i=\int_{\alpha_{i}}{\omega^{(1)}},\quad
B'_i=\int_{\beta_{i}}{\omega^{(1)}}\nonumber\\
\mbox{$\omega^{(3)}$ is of the third kind. \ \ \ 
$\omega^{(1)}$ is of the first kind.}\nonumber
\end{gather}
Applying this formula for $\omega_\mathrm{f}$ (of third kind) 
and $\omega_i$ (of first kind)
$(i=1,\dots,g)$,
we obtain
\begin{equation}
\textstyle
\int_{\beta_i}{\omega_\mathrm{f}}-o(\ee^0)\cdot\sum_{l=1}^g{(\int_{\beta_l}{\omega_i})}
=(1+o(\ee^0))(\int_{P_-}^{P_+}{\omega_i}).
\end{equation}
due to Sublemma 1 and 2
($\because$ $A'_i=\delta_{i,j}$, $A_i=o(\ee^0)$).
On the other hand, applying the formula for $\omega_{P_+-P_-}$ and $\omega_i$,
we obtain
\begin{equation}
\textstyle
\int_{\beta_i}{\omega_{P_+-P_-}}=\int_{P_-}^{P_+}{\omega_i}.
\end{equation}
Thus, we derive
\[
\textstyle
\int_{\beta_i}{({\omega_\mathrm{f}}-\{1+o(\ee^0)\}\cdot\omega_{P_+-P_-})}=o(\ee^0)\times
\sum_{l=1}^g{(\int_{\beta_l}{\omega_i})}\quad(\forall i),
\]
which implies 
$\omega_\mathrm{f}-\{1+o(\ee^0)\}\cdot\omega_{P_+-P_-}\in\mathcal{M}$.
This relation can be rewritten as
$(1+o(\ee^0))\omega_{P_+-P_-}\in\mathcal{M}+\mathcal{F}$,
or $\omega_{P_+-P_-}\in\mathcal{M}+\mathcal{F}$.
\end{proof}

\subsubsection{Differentials associated with edges}

Next we consider \textit{differentials associated with edges}.
Let $E'\subset\mathrm{Trop}\,C$ be an edge of multiplicity $m$.
By rotation, it can be assumed that $E'$ is vertical without loss of generality. 
Denote the horizontal thickness of the vertical edge $E'$ by $w$.

Let $E'=\{(B,(1-t)Y_0+tY_1)\,\vert\,0\leq t\leq 1,Y_0<Y_1\}$ and
\[
\begin{array}{c}
\{\mbox{the ceiling of $E'$}\}\subset G_{I+1}=\{Y=\mathcal{N}_{I+1}(X)\},\\
\{\mbox{the floor of $E'$}\}\subset G_I=\{Y=\mathcal{N}_{I}(X)\}.
\end{array}
\]

By definition,
it follows that $Y_1=\mathcal{N}_{I+1}(B)$ and $Y_0=\mathcal{N}_I(B)$:
Because $E'$ is of finite length, it follows that $1\leq I\leq N-1$.
The defining polynomial $f_\ep(x,y)$ of $C_\ep$ is of the form
\[\textstyle
f_\ep(x,y)=
\sum_{i=0}^N{a_{i}(x)y^{N-i}},\quad a_{I}(x)=c\ee^Ax^n\prod_j{(x-u_j\ee^{B_j})},
\]
where $c,u_j\in \rbatu$, $A,B_j\in\QQ_{>0}$, $n\in\NN$. 
We rewrite the polynomial $a_I(x)$ as
\[\textstyle
a_I(x)=c\ee^Ax^n\prod_{j=1}^w{(x-u_j\ee^{B})}\cdot
\prod_{j>\omega}{(x-u_j\ee^{B_j})},
\]
where $w$ is the horizontal thickness of $E'$ and $B_j\neq B$ $(j>w)$.
Let $E'=E_1\amalg\cdots\amalg E_w$ be the decomposition
into edges of multiplicity one.
Note that the horizontal thickness of $E_i$ equals $1$.

We define \textit{the differential associated with
the edge} $E\equiv E_1$
by:
\begin{equation}\label{eq4.14c}
\omega_E:=\frac{\phi(x)\cdot y^{N-I-1}\,dx}{f_y},\quad\quad
\frac{\phi(x)}{a_{I}(x)}:=
\frac{1}{2\pi \ai}\frac{-1}{(x-u_{1}\ee^{B})}.
\end{equation}

Let $(x,y)=(r\ee^X,s\ee^Y)$ be a point on $C_\ep$.
Assume $(X,Y)\in G_J=\{Y=\mathcal{N}_J(X)\}$
on the edge of vertical thickness $q$.
If $J=I,I+1$,
we can use the estimation 
\begin{align*}
yf_y&=\left\{
\begin{array}{ll}
(N-I+q)\,a_{I-q}y^{N-I+q}+(N-I)\,a_Iy^{N-I}+\cdots&(J=I)\\
(N-I-q)\,a_{I+q}y^{N-I-q}+(N-I)\,a_Iy^{N-I}+\cdots&(J=I+1)\\
\end{array}
\right.\\
&=\left\{
\begin{array}{ll}
-q\,a_Iy^{N-I}+\cdots&(J=I)\\
q\,a_Iy^{N-I}+\cdots&(J=I+1)\\
\end{array}
\right..
\end{align*}
This implies that
\begin{equation}\label{eq4.14}
\omega_E\sim
\left\{
\begin{array}{ll}
\{-\phi(x)/q\,a_I\}\,dx&\mbox{ if } (X,Y)\in G_I\\
\{\phi(x)/q\,a_I\}\,dx
&\mbox{ if } (X,Y)\in G_{I+1}
\end{array}\right..
\end{equation}
If $J\neq I,I+1$, $a_I\,y^{N-I}$ cannot be a dominant term of $yf_y$, and
it follows that 
\begin{equation}\label{eq4.19}
\omega_E=o(\ee^0)\,dx\qquad\mbox{if}\qquad 
(X,Y)\in G_{J}\ (J\neq I,I+1).
\end{equation}

Next we study the estimation of $\omega_E$ on vertical edges.
Let $M\subset\mathrm{Trop}\,C$ be a vertical edge whose 
ceiling is contained in
$G_{J+1}$ and whose
floor is contained in
$G_J$. Let us rewrite the expression of $\omega_E$ into
\[
\omega_E=-\phi(x)y^{N-I-1}\frac{dy}{f_x},\quad
\phi(x)=\frac{-1}{2\pi \ai}\cdot c\ee^Ax^n\textstyle
\prod_{j\neq 1}{(x-u_j\ee^{B_j})}.
\]
($\because$ $\frac{dx}{f_y}=-\frac{dy}{f_x}$ for smooth curve $C_\ep$).

Due to the definition of $M$,
the dominant term
of $f_\ep(x,y)$ on 
\[\{(x,y)\,\vert\,(\val(x),\val(y))\in M\}\]
is $a_J\,y^{N-J}$.
We claim that
the dominant term of $f_x$ is $a_J'\,y^{N-J}$, where 
$a'_i(x):=\frac{d}{dx}a_i(x)$ and
$f_x=\sum_i{a_i'y^{N-i}}$.

\begin{lemma}
The dominant term of $f_x$ on $M$ is $a_J'\,y^{N-J}$.
\end{lemma}
\begin{proof}
We consider the difference $\val(a_i)-\val(a'_i)$ $(1\leq i\leq N)$.
Let 
\[
M=\{(B_0,(1-t)Y_2+tY_3)\,\vert\,0\leq t\leq 1,Y_2<Y_3\},
\]
and $a_i=c\ee^{A_i}x^{n_i}\prod_j{(x-u_{i,j}\ee^{B_{i,j}})}$.
Then, $a'_i$ is of the form:
\[\textstyle
a'_i=
cn_i\ee^{A_i}x^{n_i-1}\prod_j{(x-u_{i,j}\ee^{B_{i,j}})}
+c\ee^{A_i}x^{n_i}\sum_k\prod_{j\neq k}{(x-u_{i,j}\ee^{B_{i,j}})}.
\]
From this expression we see that
\begin{equation}\label{eq4.15}
\val(a_i(x))-\val(a'_i(x))\leq \val(x).
\end{equation}
On the other hand, if $i=J$,
$a_J(x)$ is of the form 
\[\textstyle
a_J=c_J\ee^{A_J}x^{n_J}\prod_j{(x-u_{J,j}\ee^{B_{J,j}})},
\quad\sharp\{j\,\vert\,B_{0}=B_{J,j}\}>0.
\]
Let $\Lambda:=\{j\,\vert\,B_0=B_{J,j}\}$.
We can
assume $\Lambda=\{1,2,\dots,w'\}$ by exchanging the indices
if necessary.
Again we consider the derivative $a'_J$.
Among the factors
\[
x-u_{J,1}\ee^{B_0},\ x-u_{J,2}\ee^{B_0},\ \dots,\ 
x-u_{J,w'}\ee^{B_0},
\]
$(x-u_{J,k}\ee^{B_0})$ is the only one
that becomes very small when we take $x=u_{J,k}\ee^{B_0}+o(\ee^{B_0})$
because of the genericness condition in Section \ref{sec3}.
Thus, for fixed $k\in\Lambda$, it follows that
\begin{equation}\label{eq4.16a}
x=u_{J,k}\ee^{B_0}+o(\ee^{B_0})\ \Rightarrow\ 
a'_J(x)\sim c_J\ee^{A_J}x^{n_J}
\prod_{j\neq k}{(x-u_{J,j}\ee^{B_{J,j}})}
\end{equation}
which implies 
\begin{equation}\label{eq4.16}
\val(a_J(x))-\val(a'_J(x))=B_0=\val(x).
\end{equation}
Therefore, (\ref{eq4.15}) and (\ref{eq4.16}) give rise to
\begin{align*}
&
(\val(x),\val(y))\in M\ \Rightarrow\ 
\left\{
\begin{array}{ll}
x=u_{J,k}\ee^{B_0}+o(\ee^{B_0})\quad (\exists k)\\[1.5mm]
\val(a_Jy^{N-J})\leq \val(a_iy^{N-i})
\end{array}\right.\\
&
\ \ \Rightarrow \ \val(a'_Jy^{N-J})-\val(a'_iy^{N-i})\\
&\ \ \ \ \ \ \ \ \leq
-\val(x)+\val(a'_Jy^{N-J})+\val(x)-\val(a'_iy^{N-i})\leq 0
\end{align*}
for each $i=0,1,\dots,N$.
In particular, we can conclude that $f_x\sim a_J'y^{N-J}$ on $M$.
\end{proof}
\bigskip

Recall that the floor of $E$ is contained in $G_I$, and that
the floor of $M$ is contained in $G_J$.
From the explicit form of $a'_J$ given in (\ref{eq4.16a})
and the definition of $\phi(x)$ (\ref{eq4.14c}),
we obtain 
\begin{align*}
-&\left. \phi(x)y^{N-J}/f_x\right\vert_{(\val(x),\val(y))\in M}\\
&=\left\{\!\!
\begin{array}{ll}
(2\pi \ai)^{-1}
+\!o(\ee^0)\! & (J=I \mbox{ and } x=u_1\ee^B+o(\ee^B))\\
o(\ee^0) & (J=I \mbox{ and } x=u_j\ee^B+o(\ee^B),\ j>1)\\
o(\ee^0) & (J\neq I)
\end{array}
\right..
\end{align*}
These relations lead to the following:
\begin{equation}\label{eq4.18}
\omega_E\sim
\left\{
\begin{array}{ll}
(2\pi \ai)^{-1}(dy/y)&(\val(x),\val(y))\in E\\
o(\ee^0)\,dx&(\val(x),\val(y))\in M\neq E
\end{array}
\right..
\end{equation}
\bigskip

The three equations (\ref{eq4.14}), (\ref{eq4.19})
and (\ref{eq4.18}) gives us the singularities
of $\omega_E$.
Let $L_1$ be the leftmost leaf of $G_I$ and
$L_2$ be the leftmost leaf of $G_{I+1}$.
Denote the vertical thickness of $L_i$ $(i=1,2)$ by $q_i$ and the multiplicity by $m_i$,
respectively.
Let us consider the decomposition $L_i=L_{i,1}\amalg L_{i,2}\amalg\cdots\amalg
L_{i,m_i}$ and denote the vertical thickness of $L_{i,j}$ by $q_{i,j}$ $(q_i=q_{i,1}+\dots+q_{i,m_i})$,
the horn associated with $L_{i,j}$ by $\Sigma_{i,j}$,
the vertex
that is the end point of $L_i$ by $v_i$ and
the sphere associated with $v_i$ by $\Omega_i$.

The set $(\{x=\infty,0\}\cup\{y=\infty,0\})\cap \Sigma_{i,j}$ has only one element,
which we denote by $P_{i,j}$.
Consider cycles $\gamma_{i,j}\in\Omega_i$ which loop around 
the point $P_{i,j}$ anti-clockwise.
By (\ref{eq4.14})
it then follows that
\begin{gather}
\textstyle
\int_{\gamma_{1,j}}{\omega_E}=+(q_{1,j}/q_1)+o(\ee^0),\\
\textstyle
\int_{\gamma_{2,j}}{\omega_E}=-(q_{2,j}/q_2)+o(\ee^0).
\end{gather}
Hence, we obtain the following:
\begin{prop}\label{prop4.6}
The differential 
$\omega_E$
has a pole with residue $+q_{1,j}/(2q_1\pi \ai)$ at $P_{1,j}$
and a pole with residue $-q_{2,j}/(2q_2\pi \ai)$ at $P_{2,j}$.
\end{prop}
\bigskip

Let $\gamma\in H_1(C_\ep;\ZZ)$ be a cycle which loops a cylinder $\Sigma$.
We fix the direction of $\gamma$ as Figure \ref{fig10b}.
\begin{figure}[htbp]
\begin{center}
\includegraphics[bb=86 670 452 790,clip,width=7cm]
{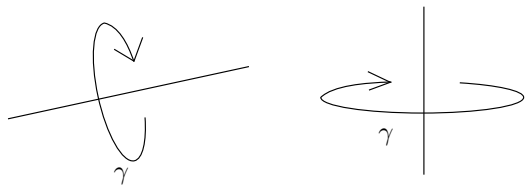}
\begin{picture}(0,0)
\put(-160,0){$\gamma$}
\put(-60,15){$\gamma$}
\end{picture}
\end{center}
\caption{We fix
the direction of $\gamma$ as these figures.
The figure on the left shows the direction of $\gamma$
for non-vertical $\Sigma$, and the one on the right shows that for vertical $\Sigma$.}
\label{fig10b}
\end{figure}
The integral 
$\int_{\gamma}{\omega_E}$
takes various values depending on the position of $\Sigma$
in $C_\ep$.
Let $M$ be a tropical edge associated with a cylinder $\Sigma$.

(i) The case $M=E$.

When one runs around a cylinder $\Sigma$, the $y$-coordinate runs around the origin.
Then, by (\ref{eq4.18}), we derive
\begin{align}
\textstyle
\int_{\gamma}{\omega_E}&=\textstyle \oint_{0+)}{(2\pi \ai)^{-1}\frac{dy}{y}}+o(\ee^0)
=\int_0^{2\pi}{(2\pi \ai)^{-1}\ai\,d\theta}+o(\ee^0)\nonumber\\
&=1+o(\ee^0).\label{eq4.26a}
\end{align}

(ii) In the case that $M$ is vertical and $M\neq E$,
it follows that $\int_\gamma{\omega_E}=o(\ee^0)$.

(iii) When $M\subset G_I\setminus
G_{I+1}$,
we consider the edge $M'=M_1\amalg M_2\amalg\cdots\amalg M_m$ $(M=M_1)$ and
denote the vertical thickness of $M_i$ by $q_i$ $(q_i=q/m)$.
From (\ref{eq4.3}) and (\ref{eq4.14}) it follows that
\begin{equation}\label{eq4.27a}
\textstyle
\int_{\gamma}{\omega_E}=
\left\{
\begin{array}{ll}
q_1/q+o(\ee^0)=1/m+o(\ee^0) & (M\subset\{X<B\})\\
o(\ee^0) & (M\subset\{X>B\})
\end{array}
\right..
\end{equation}
We used the assumption that $C$ has a good tropicalization for the first 
equality.

(iv) When $M\subset G_{I+1}\setminus
G_I$,
one has that
\begin{equation}\label{eq4.28a}
\textstyle
\int_{\gamma}{\omega_E}=
\left\{
\begin{array}{ll}
-1/m+o(\ee^0) & (M\subset\{X<B\})\\
o(\ee^0) & (M\subset\{X>B\})
\end{array}
\right. .
\end{equation}

(v) When $M\subset G_i$ $(i\neq I,I+1)$,
one has that $\int_\gamma{\omega_E}=o(\ee^0)$.
\bigskip

The remaining case is the degenerate case: $G_I\cap
G_{I+1}\neq\emptyset$. (Figure \ref{fig10})
\bigskip

(vi) The case $M\subset G_I\cap
G_{I+1}$.

Since $a_Iy^{N-I}$ is not a dominant term
in $f_y$, it follows that
\begin{equation}\label{eq4.29a}
\textstyle
\int_{\gamma}{\omega_E}=o(\ee^0).
\end{equation}
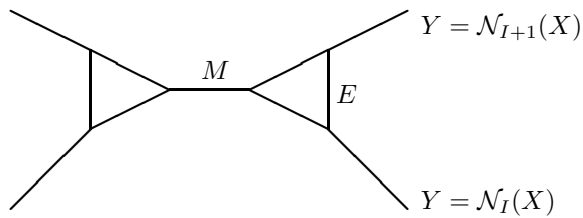
\begin{figure}[htbp]
\begin{center}
\begin{picture}(150,75)
\thicklines
\put(0,0){\line(1,1){30}}
\put(30,30){\line(2,1){30}}
\put(60,45){\line(1,0){30}}
\put(90,45){\line(2,-1){30}}
\put(120,30){\line(1,-1){30}}
\put(0,75){\line(2,-1){60}}
\put(90,45){\line(2,1){30}}
\put(120,60){\line(2,1){30}}
\put(30,30){\line(0,1){30}}
\put(120,30){\line(0,1){30}}
\put(155,0){$Y=\mathcal{N}_I(X)$}
\put(155,66){$Y=\mathcal{N}_{I+1}(X)$}
\put(123,40){$E$}
\put(72,49){$M$}
\end{picture}
\end{center}
\caption{Two subsets $G_I=\{Y=\mathcal{N}_I(X)\}$ and 
$G_{I+1}=\{Y=\mathcal{N}_{I+1}(X)\}$
may have intersection.
The edge $M$ is in the intersection.}
\label{fig10}
\end{figure}
\bigskip

Now we proceed for the $\beta$-cycle of $\omega_E$.
For this, we first calculate the integrals along the cylinders.
Let $\Sigma\subset C_\ep$ be a cylinder.
Take a path $\rho$ which runs along $\Sigma$.
We calculate the integral $\int_{\rho}{\omega_E}$
by using (\ref{eq4.14}) and (\ref{eq4.18}).
\begin{lemma}\label{lemma4.4}
Let $\mathfrak{I}:=\int_{\rho}{\omega_E}$. $(\rho$ runs along $\Sigma)$.
\begin{enumerate}
\item If $\Sigma$ is associated with 
$E=\{(B,(1-t)Y_0+tY_1)\,\vert\,0\leq t\leq 1,Y_0<Y_1\}$,
then
$\mathfrak{I}=-(2\pi \ai\ep)^{-1}(Y_1-Y_0)+o(\ee^0)$.
\item If $\Sigma$ is associated with a vertical edge except $E$, $\mathfrak{I}=o(\ee^0)$.
\item If $\Sigma$ is associated with a non-vertical edge $L$ of 
multiplicity $m$ and of
vertical thickness $q$
in $G_I$:
\begin{align*}
L=\{((1-t)X_0+tX_1,(1-t)Y_0+tY_1)\,\vert\,0\leq t\leq 1,X_0<X_1,Y_0<Y_1\}
\subset G_I,
\end{align*}
then
$\mathfrak{I}=-(2q\pi \ai\ep)^{-1}(\min{[B,X_1]}-\min{[B,X_0]})+o(\ee^0)$.
\item If $\Sigma$ is associated with a non-vertical edge $L$ of 
multiplicity $m$ and of
vertical thickness $q$
in $G_{I+1}$:
\begin{align*}
L=\{((1-t)X_0+tX_1,(1-t)Y_0+tY_1)\,\vert\,0\leq t\leq 1,X_0<X_1,Y_0<Y_1\}
\subset G_{I+1},
\end{align*}
then
$\mathfrak{I}=+(2q\pi \ai\ep)^{-1}(\min{[B,X_1]}-\min{[B,X_0]})+o(\ee^0)$.
\item
It $\Sigma$ is associated with a non-vertical edge in $G_J$
$(J\neq I,I+1)$, then $\mathfrak{I}=o(\ee^0)$.
\item If $\Sigma$ is associated with an edge in 
$G_I\cap G_{I+1}$,
then $\mathfrak{I}=o(\ee^0)$.
\end{enumerate}
\end{lemma}
\begin{proof}
1. From (\ref{eq4.18}) it follows that
\begin{align*}
\textstyle
\mathfrak{I}&\sim\textstyle
(2\pi \ai)^{-1}\cdot\int_{s_0\ee^{Y_0}}^{s_1\ee^{Y_1}}{(dy/y)}
=(2\pi \ai)^{-1}\log{\{(s_1/s_0)\ee^{Y_1-Y_0}\}}\\
&=-(2\pi \ai\ep)^{-1}(Y_1-Y_0)+\cdots.
\end{align*}
2. This can be obtained from (\ref{eq4.18}).\\
3. From (\ref{eq4.14}), 
\begin{align*}
\textstyle
\mathfrak{I}&\textstyle\sim
(2\pi \ai)^{-1}\cdot\int_{r_0\ee^{X_0}}^{r_1\ee^{X_1}}{\{1/q(x-u_{j_0}\ee^{B})\}dx}\\
&=(2q\pi \ai)^{-1}\log{\{(r_1\ee^{X_1}-u_{j_0}\ee^{B})/(r_0\ee^{X_0}-u_{j_0}\ee^{B})\}}\\
&=-(2q\pi \ai\ep)^{-1}(\min{[B,X_1]}-\min{[B,X_0]})+\cdots.
\end{align*}
4. This can be obtained in the same way as the case 3.\\
5. This follows from (\ref{eq4.19}). \\
6. This follows from (\ref{eq4.29a}).
\end{proof}
\bigskip

The result of Lemma \ref{lemma4.4} is easily understood by means of the
tropical bilinear form.
Let $\Gamma_E\subset \mathrm{Trop}\,C$ be a path 
with direction
defined by the following route:
\begin{center}
$(X=-\infty)\stackrel{\mbox{on }G_I}{\to} 
(X=B,Y=Y_0)\stackrel{\mbox{on }E}{\to}$
$(X=B,Y=Y_1)\stackrel{\mbox{on }G_{I+1}}{\to}(X=-\infty)$.
\end{center}
Using Lemma \ref{lemma4.1},
we can restate the claim of Lemma \ref{lemma4.4} as:
\begin{equation}\label{eq4.24}
\textstyle
\mathfrak{I}=
\int_\rho{\omega_E}=
-(2\pi \ai\ep)^{-1}\cdot m_L^{-1}\cdot\ell_T(L,\Gamma_E),
\end{equation}
where $m_L$ is the multiplicity of $L$.
Recall that the bilinear form $\ell_T(\cdot,\cdot)$ gives the tropical length
of intersection up to sign.
\bigskip

In fact, the equations (\ref{eq4.26a}--\ref{eq4.29a})
can be rewritten using the \textit{intersection number}.
Let $\gamma\subset C_\ep$ be a closed path surrounding some cylinder and
$E\subset\mathrm{Trop}\,C$ be a directed edge.
Denote the cylinder associated with $E$ by $\Sigma_E$.
And define the intersection number $(\gamma\circ E)$ by
\[
(\gamma\circ E):=\left\{
\begin{array}{ll}
+1& (\mbox{$\gamma$ surrounds the cylinder $\Sigma_E$ by positive direction})\\
-1& (\mbox{$\gamma$ surrounds the cylinder $\Sigma_E$ by negative direction})\\
0&(\mbox{else})
\end{array}\right..
\]
(Cf. figure \ref{fig12}).
\begin{figure}[htbp]
\begin{center}
\includegraphics[bb=100 675 480 795,clip,width=7cm]
{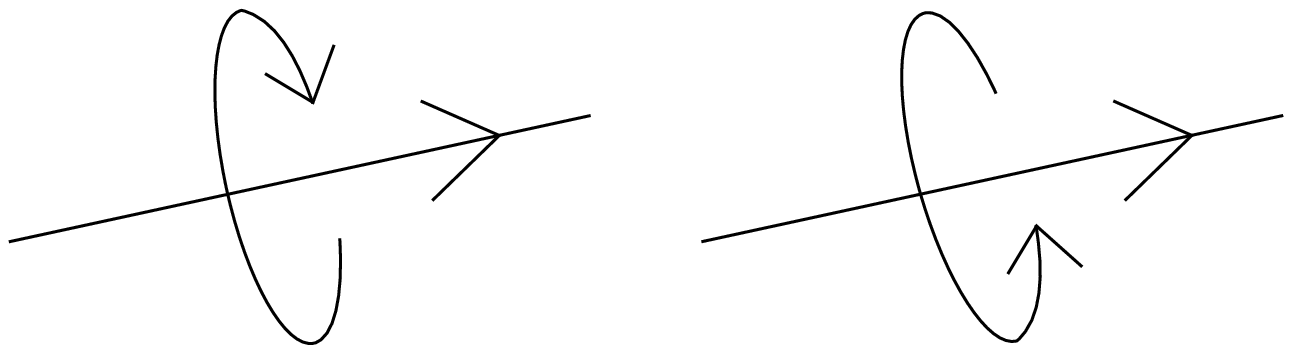}
\begin{picture}(0,0)(0,0)
\put(-200,5){$\Sigma_E$}
\put(-90,7){$\Sigma_E$}
\put(-175,35){$\gamma$}
\put(-70,35){$\gamma$}
\put(-170,-10){$(\gamma\circ E)=1$}
\put(-70,-10){$(\gamma\circ E)=-1$}
\end{picture}
\end{center}
\caption{The definition of the intersection number $(\gamma\circ E)$.}
\label{fig12}
\end{figure}

When $\gamma\subset C_\ep$ is a closed path loops a cylinder $\Sigma$,
it follows that
\begin{equation}\label{eq4.31a}
\textstyle
\int_{\gamma}{\omega_E}=m^{-1}(\gamma\circ\Gamma_E)+o(\ee^0),
\end{equation}
where $m$ is the multiplicity of the edge associated with $\Sigma$.

\subsubsection*{Poles without residue}

The differential $\omega_E$ also has poles without residue
and we can neglect the influence of these poles.
In fact, by (\ref{eq4.26a}--\ref{eq4.29a}), we can show that
$
\textstyle
\int_{\gamma}{x^k\omega_E}=o(\ee^0)\ (k\leq -1,\gamma\subset\{\val(x)<\!\!<0\}),
$
which implies
\begin{align*}
\omega_E=\left(\frac{c_{-n}}{z^n}+\cdots+\frac{c_{-1}}{z}+c_0+c_1z+\cdots\right)dz
\mbox{\quad at } z\in\{x=\infty\}\\
\ \Rightarrow\ c_k=o(\ee^0)\ \ (k<-1).
\end{align*}

\subsubsection{A modified differential}

Let $E=E_1\amalg\dots\amalg E_m$ be an edge of multiplicity $m$.
Define the differential $\upsilon:=\omega_{E_i}-\omega_{E_j}$, where 
$\omega_{E_i}$ is the differential associated with the edge $E_i$.
Then, $\upsilon$ satisfies the following:\\
(i) $\upsilon$ has no pole with residue, \\
(ii) For a closed path $\gamma$
surrounding a cylinder $\Sigma$ in $C_\ep$, one has that
\[\textstyle
\int_\gamma{\upsilon}=
(\gamma\circ (E_i-E_j))+o(\ee^0).
\]
(iii) For a path $\rho\subset C_\ep$ 
which runs along a cylinder $\Sigma$, it follows that
\[\textstyle
\int_\rho{\upsilon}=
-(2\pi \ai\ep)^{-1}\cdot\ell_T(L,E_i-E_j),
\]
where $L$ is the edge which is associated with $\Sigma$ and which is of multiplicity one.

Now we define a new differential which is a modification of $\omega_E$.
Let $\Gamma_E=E\amalg E^{(1)}\amalg E^{(2)}\amalg\dots\amalg 
E^{(n)}$ be the decomposition into edges.
We decompose each $E^{(i)}$ into edges of multiplicity one:
$E^{(i)}=E^{(i)}_1\amalg\cdots\amalg E_{m_i}^{(i)}$, where $m_i$ is the multiplicity of 
$E^{(i)}$.

Now we define a new modified differential associated with $E$.
Let 
\begin{gather*}
\tilde\omega_E:=\omega_E+\upsilon^{(1)}+\upsilon^{(2)}+\dots+\upsilon^{(n)},\\
\upsilon^{(i)}=m_i^{-1}\left\{(\omega_{E^{(i)}_1}-\omega_{E^{(i)}_2})
+(\omega_{E^{(i)}_1}-\omega_{E^{(i)}_3})+\dots+(\omega_{E^{(i)}_1}-
\omega_{E^{(i)}_{m_i}})
\right\}.
\end{gather*}
For the path $\tilde\Gamma_E$ which is defined by
$
\tilde\Gamma_E:=E\amalg E_1^{(1)}\amalg E_1^{(2)}\amalg\dots\amalg
E_1^{(n)}$,
we can rewrite
the equation (\ref{eq4.24}) as
\begin{equation}\label{eq4.30}
\textstyle
\int_\rho{\tilde\omega_E}=-(2\pi \ai\ep)^{-1}\ell_T(L,\tilde\Gamma_E),\quad
(\mbox{$\rho$ runs along $L$}),
\end{equation}
and we also rewrite the equation (\ref{eq4.31a}) as
\begin{equation}\label{eq4.31}
\textstyle
\int_\gamma{\tilde\omega_E}=(\gamma\circ\tilde\Gamma_E)+o(\ee^0).
\end{equation}

\subsubsection{Proof of the theorem}

We have finished all the preparations necessary to complete the proof of the main theorem.
Let $\mathfrak{X}$ be a set of edges contained in $\mathrm{Trop}\,C$.
Denote the free additive abelian group 
which is generated by the elements of $\mathfrak{X}$
by $\ZZ_{\mathfrak{X}}$. 

For a closed path $\Gamma$ contained in $\mathrm{Trop}\,C$,
choose edges $E_1,\dots,E_n;F_1,\dots,F_m$ of multiplicity $1$
such that $\tilde\Gamma_{E_1}
+\cdots+\tilde\Gamma_{E_n}
-\tilde{\Gamma}_{F_1}-\dots-\tilde{\Gamma}_{F_m}
=\Gamma\in\ZZ_{\mathfrak{X}}$.
Define the differential 
$$\omega_\Gamma':=
\tilde\omega_{E_1}
+\cdots+\tilde\omega_{E_n}-\tilde\omega_{F_1}-
-\cdots-\tilde\omega_{F_m}.
$$
By Proposition \ref{prop4.6},
(\ref{eq4.30}) and (\ref{eq4.31}), 
$\omega_E'$ has the following properties:
\begin{enumerate}
\def\labelenumi{(\theenumi)}
\def\theenumi{\roman{enumi}}
\item $\omega_E'$ has singularities in $\{x=\infty,0\}\cup\{y=\infty,0\}$.
\item Let $P_1,\dots,P_q$ be points in $\{x=\infty,0\}\cup\{y=\infty,0\}$
and suppose that these are associated with the same leaf of $\mathrm{Trop}\,C$.
Then, $\sum_{i=1}^q{\mathrm{Res}_{P_i}(\omega_\Gamma')}=0$.
\item Let $\alpha$ be a closed path surrounding a cylinder $\Sigma$
which is associated with the edge
$E\subset\mathrm{Trop}\,C$.
Then,
\begin{equation}\label{eq4.25}
\textstyle
\int_{\alpha}{\omega_\Gamma'}=
(\alpha\circ\Gamma)+o(\ee^0).
\end{equation}
\end{enumerate}
For a leaf $\Gamma_\infty$
of infinite length in $\mathrm{Trop}\,C$, 
we define the differential $\omega_{\Gamma_\infty}$
by:
\[
\omega_{\Gamma_\infty}:=
\omega_{c_1P_1+\cdots+c_nP_n},\quad
\left(
\begin{array}{c}
\text{$P_i$ is a point in the horns associated
with $\Gamma_\infty$}\\
\text{ s.t. 
$c_i=\mathrm{Res}_{P_i}(\omega'_\Gamma)$ is not $0$.
}
\end{array}
\right).
\]
(Recall
`$\omega_{c_1P_1+\dots+c_nP_n}$' is the normalised differential of the third kind.) 
Summing these differentials of the third kind for all edges of infinite length:
\[
\omega_\infty:=\sum_{\vert\Gamma_\infty\vert=\infty}{\omega_{\Gamma_\infty}}.
\]
Then, the new differential $\omega_\Gamma:=\omega_\Gamma'-\omega_\infty$
satisfies:
i) $\int_\alpha{\omega_\Gamma}=\int_\alpha{\omega_\Gamma'}$ and
ii)
$\omega_\Gamma$ has no singularity with non-zero residue
($\omega_\Gamma$ is of the second kind).
Moreover, by adding the normalised differentials of the second kind,
we can assume $\omega_\Gamma$ is of first kind.
(Recall that we can neglect the poles without residue.)

\begin{lemma}\label{lastlemma}
$\omega'_\Gamma-\omega_\Gamma\in\mathcal{M}+\mathcal{F}$.
\end{lemma}
\begin{proof}
Due to Lemma \ref{lemma4.2}, it follows that $\omega_{\Gamma_\infty}\in\mathcal{M}
+\mathcal{F}$ for each leaf $\Gamma_\infty$.
Because $\mathcal{M}$ and $\mathcal{F}$ are vector spaces, the required result 
is obtained soon. 
\end{proof}

\textbf{Proof of the theorem.} 
Let $\Gamma=T_{\beta_i}$, that is the closed path on $\mathrm{Trop}\,C$
associated with the $\beta$-cycle $\beta_i$ on $C_\ep$.
(See Section \ref{sec4.3}.)
By (\ref{eq4.25}), it follows that
\begin{equation}\label{lastequation}
\textstyle
\int_{\alpha_j}{\omega_{T_{\beta_i}}}=
\int_{\alpha_j}{\omega_{T_{\beta_i}}'}=
\left\{
\begin{array}{ll}
1+o(\ee^0)& (i=j)\\
o(\ee^0)&(i\neq j)
\end{array}
.\right.
\end{equation}

Let $\omega_j$ be the $j$-th normalised holomorphic differential 
(Section \ref{sec4.3}).
Clearly, (\ref{lastequation}) means 
\[\omega_j=\omega_{T_{\beta_j}}\cdot(1+o(\ee^0)),\quad \forall j,\]
or equivalently $\omega_j-\omega_{T_{\beta_j}}\in\mathcal{M}$.
Due to Lemma \ref{lastlemma} 
we obtain 
\begin{equation}\label{eq4.27}
\omega_i-\omega'_{T_{\beta_i}}\in\mathcal{M}+\mathcal{F}.
\end{equation}

Consider $g\times g$ matrices $B=(\int_{\beta_j}{\omega_i})_{i,j}$ and
$B'=(\int_{\beta_j}{\omega'_{T_{\beta_i}}})_{i,j}$.
Equation (\ref{eq4.27}) can be rewritten as
\begin{gather}\label{eq4.28}
B-B'=o(\ee^0)B+B^\dagger\qquad
\textstyle\lim_{\ep\to 0^+}{\vert B^\dagger\vert}<\infty,\\
\text{or}\qquad \{I+o(\ee^0)\cdot\Delta\}\cdot B-B'=B^\dagger,
\label{eq4.29}
\end{gather}
where $I$ is the identity matrix and $\Delta$ is a $g\times g$ matrix.
On the other hand, from (\ref{eq4.24})
one concludes that $B'$ tends to infinity when $\ep\to 0^+$.
We thus obtain
\begin{equation}
B\sim B'\quad\quad\quad(\ep\to 0^+),
\end{equation}
by taking a limit $\ep\to 0^+$ of (\ref{eq4.29}).

To conclude the proof of theorem,
it is sufficient to prove that
$$
\textstyle
\int_{\beta_j}{\omega'_{T_{\beta_i}}}=-(2\pi \ai\ep)^{-1}\cdot
\ell(T_{\beta_j},T_{\beta_i}),
$$
which is
a mere linear combination of copies of (\ref{eq4.30}). $\qed$

\ack
I am very grateful to Professor Tetsuji Tokihiro and
Professor Ralph Willox for helpful comments 
on this paper. 
I would like to thank Professor Jonathan Nimmo for valuable advice
on the revised version of this paper.
This work was supported by Grant-in-Aid for
the Japan Society for the Promotion of Science Fellows (09J07090).

\appendix
\section{Genericness Condition}

In this paper, we introduced some conditions on $C_\ep$ and $\mathrm{Trop}\,C$
to make the problem easier.
Let $f_\ep(x,y)=\sum_{i}{a_i(x)y^{N-i}}$ be the defining polynomial of $C_\ep$, where
\[\textstyle
a_i(x)=c_i\ee^{A_i}x^{m_i}\prod_{j=1}^{d_i}{(x-u_{i,j}\ee^{B_{i,j}})},
\quad
c_i,u_{i,j}\in \rbatu\!,\ A_i,B_{i,j}\in\QQ,\,m_i\in\NN.\]
Let $\theta\in SL_2(\ZZ)$ be a rotation of $\mathrm{Trop}\,C$.
The translation
$\theta$ naturally acts on $C_\ep$ by $x\mapsto x^{\delta}y^{-\beta}$;
$y\mapsto x^{-\gamma}y^\alpha$.
Define the new polynomial
\begin{gather*}
\textstyle
f^\theta(x,y)=f_\ep(x^{\delta}y^{-\beta},x^{-\gamma}y^\alpha)
=\sum_{i}{a^\theta_i(x)y^{N'-i}}\\
\textstyle
a^\theta_i(x)=c^\theta_i
\ee^{A^\theta_i}x^{m^\theta_i}\prod_{j}{(x-u^\theta_{i,j}\ee^{B^\theta_{i,j}})}.
\end{gather*}

To be precise, we assumed
three conditions:
\begin{description}
\item[\textbf{Genericness condition.}] 
For fixed $\theta\in SL_2(\ZZ)$,
$\htht (u^\theta_{i,j})\in\CC\setminus\{0\}$ $(\forall i,j)$ are all distinct.
\item[\textbf{Condition I.}]
For each edge $E$, 
$m=\mathrm{g.c.d.}(q,w)$,
where $m,q,w$ respectively are
the multiplicity, the vertical thickness and the horizontal thickness of $E$.
\item[\textbf{Condition I\!I.}] 
For each edge $E=E_1\amalg\dots\amalg E_m$,
$q_1=\dots=q_m$, $w_1=\dots=w_m$, where $q_i,w_i$ 
respectively
are the vertical thickness and
the horizontal thickness
of $E$.
\end{description}

The following relation exists between these conditions.
\begin{prop}
Genericness condition $\Rightarrow$
Condition I\!I
$\Rightarrow$ Condition I.
\end{prop}
\begin{proof}
(Genericness cond.\relax\,$\Rightarrow$ Cond.\relax{}\,I\!I)\ \ 
We first prove the case when $E=E_1\amalg\dots\amalg E_m$ is vertical.
Then it is clear that $q_1=\dots=q_m=0$.
The defining equation of 
vertical edge $E$ is of the form $(a+w)X+bY+c=aX+bY+c'$.
When we substitute $(x,y)=(r\ee^X,s\ee^Y)$, $(X,Y)\in E$ into $f_\ep(x,y)$,
the polynomial $a_{N-b}(x)y^b$ is dominant.
Moreover, we can derive the relation $\htht{(a_{N-b}(x))}=0$, which gives us:
\[\htht
(r-u_{i,j_1})(r-u_{i,j_2})\dots
(r-u_{i,j_w})=0,\quad j=N-b,\,
j_k\in\{j\,\vert\,B_{i,j}=(c-c')/w\}.
\]
By the genericness condition, 
this equation implies that $w$ distinct cylinders in $C$
are associated with the edge $E$.
Then $m=w$, which implies $w_1=w_2=\dots=w_m=1$.

In the general case, we consider $\theta\in SL_2(\ZZ)$
such that $\theta\cdot E=(\theta \cdot E_1)\amalg\cdots\amalg(\theta \cdot E_m)$ is vertical.
In fact, the vertical thickness and the horizontal thickness of $E_i$ satisfy the relation
\begin{equation}\label{eq.app1}
q_i=\alpha q_i^\theta+\beta w_i^\theta,\qquad
q_i=\gamma q_i^\theta+\delta w_i^\theta,
\end{equation}
where $q_i^\theta$ and $w_i^\theta$ are the vertical thickness and the horizontal thickness of 
$\theta\cdot E_i$.
This equation implies $q_1=\dots=q_m$ and $w_1=\dots=w_m$.

(Cond.\relax\,I\!I $\Rightarrow$ Cond.\relax\,I)\ \ 
From the equation $w_1=\dots=w_m=1$ for the vertical edge $E=E_1\amalg\dots
\amalg E_m$
we conclude that
it is enough to prove that $\mathrm{g.c.d.}(q,w)$ is invariant under the rotation $\theta$.
This fact follows immediately from (\ref{eq.app1}).
\end{proof}
\bigskip

Due to the above proposition, we can claim that the only
essential assumption for our arguments
is the genericness condition.

\end{document}